\documentclass[12pt]{amsart}
\usepackage{enumitem,amssymb,aliascnt,geometry,stmaryrd,mathtools,microtype,xltabular}
\usepackage{tikz-cd}
\usepackage{version}
\usepackage[pdfusetitle]{hyperref}

\hypersetup{ pdfkeywords={Group actions on varieties, essential dimension, Chern numbers},
hidelinks}

\usepackage{etoolbox}
\makeatletter
\pretocmd{\section}{\addtocontents{toc}{\protect\addvspace{6\p@}\bfseries}}{}{}
\pretocmd{\subsection}{\addtocontents{toc}{\protect\normalfont}}{}{}
\makeatother

\DeclareMathOperator{\Spec}{Spec}
\DeclareMathOperator{\Stab}{Stab}
\DeclareMathOperator{\Confi}{Conf}
\DeclareMathOperator{\Gal}{Gal}
\DeclareMathOperator{\trdeg}{trdeg}
\DeclareMathOperator{\ind}{ind}
\DeclareMathOperator{\Oo}{O}
\DeclareMathOperator{\ed}{ed}
\DeclareMathOperator{\Hom}{Hom}
\DeclareMathOperator{\Isom}{Isom}
\DeclareMathOperator{\End}{End}
\DeclareMathOperator{\Aut}{Aut}

\DeclareMathOperator{\CH}{CH}

\DeclareMathOperator{\Ind}{Ind}
\DeclareMathOperator{\GF}{(GF)}

\newcommand{\Conf}[2]{\Confi_{#1}({#2})}
\newcommand{\red}{\mathrm{red}}
\newcommand{\dab}{d_p^{\ab}}
\newcommand{\der}{\mathrm{der}}
\newcommand{\dabb}{d_{(p)}^{\ab}}
\newcommand{\ab}{\mathrm{ab}}
\newcommand{\sm}{\mathrm{sm}}
\newcommand{\Cr}{\mathrm{Cr}}
\newcommand{\GL}{\mathrm{GL}}
\newcommand{\PGL}{\mathrm{PGL}}
\newcommand{\Gm}{\mathbb{G}_m}

\newcommand{\id}{\mathrm{id}}
\newcommand{\stacks}[1]{\cite[Tag~\href{https://stacks.math.columbia.edu/tag/#1}{{#1}}]{stacks}}
\newcommand{\tw}[2]{{}^{#2}\!{#1}}

\newcommand{\Oc}{\mathcal{O}}
\newcommand{\Cc}{\mathcal{C}}
\newcommand{\Pc}{\mathcal{P}}
\newcommand{\Uc}{\mathcal{U}}
\newcommand{\Zz}{\mathbb{Z}}
\newcommand{\Pp}{\mathbb{P}}
\newcommand{\Nn}{\mathbb{N}}
\newcommand{\Qq}{\mathbb{Q}}
\newcommand{\Rr}{\mathbb{R}}
\newcommand{\Fp}{\mathbb{F}_p}
\newcommand{\Tan}{T}

\newtheorem{thm}{Theorem}
\newtheorem{prop}[thm]{Proposition}

\newtheorem*{que*}{Question}

\swapnumbers

\newtheorem{theorem}{Theorem}[section]
\newaliascnt{proposition}{theorem}
\newtheorem{proposition}[proposition]{Proposition}
\newaliascnt{lemma}{theorem}
\newtheorem{lemma}[lemma]{Lemma}
\newaliascnt{corollary}{theorem}
\newtheorem{corollary}[corollary]{Corollary}

\theoremstyle{definition}
\newaliascnt{remark}{theorem}
\newtheorem{remark}[theorem]{Remark}
\newaliascnt{example}{theorem}
\newtheorem{example}[example]{Example}
\newaliascnt{definition}{theorem}
\newtheorem{definition}[definition]{Definition}

\newtheoremstyle{par}
{}
{}
{}
{}
{}
{.}
{ }
{}%
\theoremstyle{par}
\newtheorem{para}[theorem]{}

\numberwithin{equation}{theorem}

\begin{document}

\begin{abstract}
	Let \(X\) be a smooth, projective, connected variety over a field \(k\) containing a root of unity of order \(p\).
	If \(X\) has a Chern number prime to \(p\), we show that every finite \(p\)-group acting on \(X\) acts through a subgroup of \(\operatorname{GL}_n(k)\), where \(n=\dim X\).
	This allows one to transfer properties of representations of finite \(p\)-groups to their actions on \(X\).
	We deduce a fixed-point theorem which, unlike earlier results of this kind, is sensitive to the arithmetic of the base field, and whose dimensional bound is sharp for each individual group (apart from one low-dimensional case when \(p=2\)).
	We also obtain restrictions on finite \(p\)-groups of birational automorphisms of rationally connected varieties;
	for instance the Cremona group \(\operatorname{Cr}_n(\mathbb{Q})\) contains no element of order \(p^2\) when \(p \ge n+2\).

	More generally, we prove variants of these results in the presence of a Chern number not divisible by a given power of \(p\).

	The method is based on the following observation about the essential \(p\)-dimension \(\operatorname{ed}_p(G)\) of an affine algebraic group \(G\):
	over a field of characteristic zero, the number \(\operatorname{ed}_p(G) + \dim G\) is the least dimension of a smooth projective variety \(Y\) with a generically free \(G\)-action such that the degree map \(\operatorname{CH}_G(Y) \to \mathbb{F}_p\) from the equivariant Chow group is nonzero.
	The Chern numbers modulo \(p\) of \(Y\) lie in the image of this map;
	combined with Karpenko and Merkurjev's computation of the essential \(p\)-dimension of finite \(p\)-groups, this yields our results.
\end{abstract}

\author{Olivier Haution}
\title{Essential \(p\)-dimension and Chern numbers}
\email{olivier.haution at unimib.it}
\address{Dipartimento di Matematica e Applicazioni, Università degli Studi di Milano-Bicocca, via Roberto Cozzi 55, 20125 Milano, Italy}


\date{\today}

\maketitle

\numberwithin{theorem}{section}
\numberwithin{lemma}{section}
\numberwithin{proposition}{section}
\numberwithin{corollary}{section}
\numberwithin{example}{section}
\numberwithin{definition}{section}
\numberwithin{remark}{section}

\section*{Introduction}
Consider a finite \(p\)-group \(G\) acting on a smooth, projective, connected variety \(X\) over a field \(k\) of characteristic different from \(p\).
For each integer \(a\ge 0\), fix a primitive \(p^a\)-th root of unity \(\zeta_{p^a}\) in an algebraic closure of \(k\).

If \(G\) has a rational fixed point on \(X\), the tangent space at that point is a \(G\)-representation of dimension \(\dim X\), which turns out to be faithful if the action on \(X\) is faithful, as one sees by linearizing the action at the fixed point.

Our main result obtains a similar conclusion under a very different hypothesis, involving the Chern numbers of \(X\).
These are the integers obtained as the degrees of polynomials in the Chern classes of the tangent bundle of \(X\).
\begin{thm}[see (\ref{cor:quotient})]
	\label{thm_main}
	Assume that \(X\) has a Chern number prime to \(p\).
	Then \(G\) acts on \(X\) through a subgroup of \(\GL_n(k(\zeta_p))\), where \(n=\dim X\).
\end{thm}
In this statement, adjoining the root of unity \(\zeta_p\) cannot be avoided (see (\ref{rem:roots_are_needed})).

The Chern numbers are often amenable to computation.
Since they are independent of the action, the theorem constrains all actions of \(p\)-groups on the variety \(X\).

The hypotheses of the theorem do not imply the existence of a fixed point, in general:
for a given \(G\), we determine in \S\ref{sect:construction} the least dimension in which they fail to guarantee a fixed point.
In smaller dimension they do guarantee one, as we deduce from the theorem (Theorem~\ref{intro:main} below).

Under the assumption that \(X\) has a Chern number prime to \(p\), the theorem allows us to transfer certain properties of \(G\)-representations to the \(G\)-action on \(X\).
For instance, every \(G\)-representation of dimension \(<p\) is abelian, in the sense that it factors through the abelianization of \(G\);
we obtain that the \(G\)-action on \(X\) is commutative when \(\dim X <p\).
More generally, representations of \(G\) of dimension \(<p^\ell\) factor through a quotient of derived length \(\le \ell\), and so the same is true for the \(G\)-action on \(X\) when \(\dim X <p^\ell\) (see (\ref{prop:derived})).

Theorem~\ref{thm_main} also has consequences of an arithmetic flavor.
In particular, it provides bounds on the group order in terms of \(p\)-primary roots of unity in \(k(\zeta_p)\): if \(p\) is odd, or if \(p=2\) with \(\zeta_4 \in k\), setting
\[
	m=\sup \{a \ge 1 \text{ such that } \zeta_{p^a} \in k(\zeta_p)\} \in \Nn \cup \{\infty\},
\]
we obtain in (\ref{prop:Minkowski}) that the \(G\)-action on \(X\) factors through a quotient \(Q\) satisfying
\[
	v_p(|Q|)\le mn+v_p(n!),\qquad n=\dim X>0,
\]
using the classical Minkowski--Schur bound in the form proved by Serre \cite{Serre-Gk}.

We also obtain that the \(G\)-action on \(X\) factors through a quotient of exponent dividing \(p^{s-1}\) when \(\dim X <[k(\zeta_{p^s}):k(\zeta_p)]\) (see (\ref{cor:exponent})).
This yields the following property of the Cremona group \(\Cr_n(k)\) of the birational automorphisms of \(\Pp^n\).
\begin{prop}[=(\ref{prop:Cremona})]
	Let \(k\) be a field of characteristic zero.
	If \(n< p-1\), then \(\Cr_n(k)\) contains no element of order \(p^s\) whenever \(\zeta_{p^s} \not \in k(\zeta_p)\).
\end{prop}
For instance, the group \(\Cr_n(\Qq)\) contains no element of order \(p^2\) when \(n<p-1\).
The mechanism here is based on the classical observation that rationally connected varieties of dimension \(<p-1\) have a Chern number prime to \(p\).\\

We explore in detail the obstruction to noncommutative actions provided by Theorem~\ref{thm_main}, mostly because of the relation with fixed-point theorems (see Theorem~\ref{intro:main}).
It turns out that a better bound than \(\dim X <p\) for imposing commutativity of the action is available if we restrict ourselves to actions of a particular group \(G\).
For instance, when \(G\) is extraspecial of order \(p^{1+2r}\), all \(G\)-representations of dimension \(<p^r\) are abelian, and as above this property transfers to the \(G\)-action on \(X\).
This leads us to introduce the following invariant (see (\ref{def:d_p}) and (\ref{prop:d_p})):
\[
	\dab(G) = \min \{\dim_{k(\zeta_p)} V, \text{ for \(V\) a nonabelian \(G\)-representation over \(k(\zeta_p)\)}\} \in \Nn \cup \{\infty\}.
\]
The invariant \(\dab(G)\) takes the value \(\infty\) when \(G\) is abelian, and is otherwise always a power of \(p\).
It can decrease upon extending the field \(k\).

For \(X\) as above, and with a Chern number prime to \(p\), we obtain the following (see (\ref{prop:d_p_constant}) and (\ref{th:P}.\ref{th:P:constant})).
If \(\dim X <\dab(G)\), then the \(G\)-action on \(X\) factors through an abelian quotient; more precisely it factors through the group \((\Zz/p^{a_1}) \times\dotsb \times (\Zz/p^{a_r})\), with \(a_1, \dotsc, a_r \ge 1\) and
\[
	[k(\zeta_{p^{a_1}}):k(\zeta_p)] + \dotsb + [k(\zeta_{p^{a_r}}):k(\zeta_p)] \le \dim X.
\]

It was observed in \cite[(1.1.2.i)]{fpt} that actions of abelian \(p\)-groups have fixed points (over some extension of the base field \(k\)) on smooth projective varieties with a Chern number prime to \(p\).
This yields the following fixed-point theorem:
\begin{thm}[see (\ref{th:fpt})]
	\label{intro:main}
	Assume that \(X\) has a Chern number prime to \(p\).
	If \(\dim X <\dab(G)\), then \(X^G\ne \varnothing\).
\end{thm}
Since \(\dab(G)\ge p\) when \(G\) is nonabelian, the above theorem may be viewed as a higher analogue of the fixed-point theorem of \cite[(1.1.2.iii)]{fpt}, which asserted that all \(p\)-groups have fixed points on varieties of dimension \(<p\) with a Chern number prime to \(p\).
An important difference is that Theorem~\ref{intro:main} takes into account the structure of the group \(G\) via the invariant \(\dab(G)\), while the bound of \cite[(1.1.2.iii)]{fpt} was uniform (namely, the minimum of the values \(\dab(G)\) over all \(p\)-groups \(G\)).
The proof of \cite[(1.1.2.iii)]{fpt} was essentially \(K\)-theoretic and specifically tailored to the case of dimension \(<p\), and thus very different from the methods of the current paper, which rely on essential dimension.\\

A striking feature of Theorem~\ref{intro:main} is its sensitivity to the arithmetic of the base field, in the sense that it provides more restrictions than could be obtained by passing to an algebraic closure before applying it.
In particular, it encodes purely algebraic phenomena having no analogs in topology.
This is in stark contrast to earlier fixed-point theorems involving Chern numbers, such as those of \cite{fpt}.

The value of the invariant \(\dab(G)\) can change upon extending the base field \(k\) for two distinct reasons: the existence of irreducible representations having nontrivial Schur index (for \(p=2\) only), and the lack of \(p\)-primary roots of unity in the field \(k\) (for any \(p\)).
For instance, there exists a \(2\)-group \(G\) such that \(d_2^{\ab}(G_\Qq)=8\) and \(d_2^{\ab}(G_{\Qq(i)})=4\) (see (\ref{ex:d_2_ext})).\\

Apart from the case \(p=2\) and \(d_2^{\ab}(G)=2\) (where the bound can actually be improved from \(2\) to \(4\)), we show in \S\ref{sect:construction} that the bound of Theorem~\ref{intro:main}  is sharp, for each individual group \(G\).
This involves constructing explicit examples of actions without fixed points on varieties of the prescribed dimension \(\dab(G)\), having a Chern number prime to \(p\).
For \(p=2\), the obvious candidate fails:
when \(V\) is a nonabelian \(G\)-representation of minimal dimension, the blow-up of the compactification \(\Pp(V\oplus 1)\) of \(V\) at the fixed locus \(\Pp(1)\) has no odd  Chern numbers, and blowing up further does not appear to help.
We use instead the Weil restriction of such a compactification along a suitably chosen quadratic étale algebra;
this device also lets us handle representations of Schur index \(2\).\\

The assumption that a Chern number of \(X\) is prime to \(p\) can be relaxed, at the cost of getting a faithful representation of a larger dimension:
\begin{thm}[see (\ref{th:GL_p-adic}) and (\ref{prop:Chern_versal}.\ref{prop:Chern_versal:1})]
	\label{thm_p-adic}
	Assume that \(X\) is geometrically connected, that \(\dim X >0\), and that \(X\) has a Chern number not divisible by \(p^{c+1}\), for some \(c\in \Nn\).
	Then \(G\) acts on \(X\) through a subgroup of \(\GL_{p^cn}(k(\zeta_p))\), where \(n=\dim X\).
\end{thm}
Roughly speaking, this result is obtained by applying the same methods to the configuration space parametrizing degree \(p^c\) étale subschemes of \(X\), instead of \(X\) itself.
This yields generalizations of all the above results, with the exception of the fixed-point theorem.
For instance, when \(k\) has characteristic zero, the Cremona group \(\Cr_n(k)\) contains no element of order \(p^s\), whenever
\[
	[k(\zeta_{p^s}):k(\zeta_p)] > n \cdot p^{\lfloor n/(p-1) \rfloor},
\]
which reads \(s > 1 + \lfloor n/(p-1) \rfloor + \lfloor \log_pn \rfloor\) for \(k=\Qq\) (see (\ref{rem:cr_p-adic})).
To the best of our knowledge, this is the first explicit bound of this kind valid in every dimension.\\

The relevant tool to prove Theorem~\ref{thm_main} is the concept of essential dimension, or rather its \(p\)-local variant, the essential \(p\)-dimension \(\ed_p(G)\) of an algebraic group \(G\) over \(k\).
This invariant encodes the arithmetic complexity of \(G\)-torsors over extensions of the field \(k\).

It turns out that a smooth projective variety \(X\) with a \(G\)-action is \emph{\(p\)-versal} if and only if the  (modulo \(p\)) degree map from Edidin--Graham's equivariant Chow group
\[
	\CH_G(X) \to \CH(X) \xrightarrow{\deg} \Zz \to \Fp
\]
is nonzero.
Under this assumption, the action on \(X\) can be generically free only when \(\dim X \ge \ed_p(G)+\dim G\).

As an aside, this point of view can actually be used as a definition of essential \(p\)-dimension, at least when resolution of singularities is available:
\begin{prop}[=(\ref{prop:ed_smooth_proj})]
	Let \(G\) be an affine group scheme of finite type over \(k\), and assume that \(k\) has characteristic zero.
	Then the number \(\ed_p(G)+\dim G\) is the least dimension of a smooth projective \(k\)-variety \(X\) with a generically free \(G\)-action such that the degree map \(\CH_G(X) \to \Fp\) is nonzero.
\end{prop}
This observation makes it possible to approach essential \(p\)-dimension from an intersection-theoretic perspective.
In particular, as Chern numbers modulo \(p\) belong to the image of the degree map, they can obstruct generic freeness of actions when the dimension of the variety is too low.
This is ultimately the mechanism behind the properties of actions of \(p\)-groups discussed above, using as input the computation by Karpenko and Merkurjev of the essential \(p\)-dimension of finite \(p\)-groups \cite{KM-Ess}.

Something is lost, of course, by considering only Chern numbers instead of the full image of the degree map, but this loss is precisely what makes the obstruction applicable in practice, as these numbers do not depend on the action.
It is this independence that allows us to formulate a fixed-point theorem valid for all actions of finite \(p\)-groups simultaneously, based on nonequivariant information about the variety.\\

Our methods are decidedly algebraic, as they rely on the concept of essential dimension, and we do not know whether an analog holds in differential topology:
\begin{que*}
	Let \(G\) be a finite \(p\)-group acting smoothly on a closed almost complex manifold \(X\) of complex dimension \(n\), preserving the almost complex structure, and suppose some Chern number of \(X\) is prime to \(p\).
	Is \(X^G\) nonempty when \(n\) is smaller than the least dimension of a nonabelian complex \(G\)-representation?
\end{que*}

Finally, let us mention that a similar strategy works for groups of multiplicative type, rather than finite \(p\)-groups, using the computation of their essential \(p\)-dimension in \cite{LMMR-tori}.
We plan to explore this topic further in a subsequent paper.

\subsection*{Outline of the paper.} 
In \S\ref{sect:versality}, we set up notation and recall some facts about versality of actions.
We introduce the notion of \(G\)-index.

In \S\ref{sect:essential}, we recall the definition of essential \(p\)-dimension, and establish the description in terms of the equivariant Chow group, which will be useful in the subsequent sections.

In \S\ref{sect:Chern}, we relate Chern numbers modulo \(p\) to essential \(p\)-dimension.
This allows us to obtain Theorem~\ref{thm_main}.

In \S\ref{sect:noncommutative}, we deduce dimensional bounds for faithful actions of abelian \(p\)-groups, expressed in terms of cyclotomic degrees.

In \S\ref{sect:MS}, we deduce the Minkowski--Schur bounds on the group order.

In \S\ref{sect:obstruction}, we describe a general procedure to  show that actions on varieties with a Chern number prime to \(p\) factor through a group in a given class.

In \S\ref{sect:fpt}, we specialize to obstructing nonabelian actions, using the invariant \(\dab\) mentioned above.
Its definition is tailored to the statement, which then becomes mostly formal;
the formula in terms of dimensions of representations given in \S\ref{sect:obstruction} is a direct consequence of Karpenko and Merkurjev's computation of the essential \(p\)-dimension of finite \(p\)-groups.
We deduce the fixed-point theorem, Theorem~\ref{intro:main} above.

In \S\ref{sect:birat}, we apply these results to birational automorphism groups of rationally connected varieties of low dimension (\(<p-1)\).

In \S\ref{sect:derived}, we deduce a theorem restricting the derived length of \(p\)-groups acting faithfully on varieties with a Chern number prime to \(p\).

In \S\ref{sect:construction}, we construct the varieties with a Chern number prime to \(p\) carrying a fixed-point-free action of a given group \(G\), and realizing the smallest dimension \(\dab(G)\) allowed by Theorem~\ref{intro:main}.

In \S\ref{sect:conf}, we study the configuration space parametrizing finite étale subschemes of a fixed degree in a given variety.
Applying the method of \S\ref{sect:Chern} to the configuration space yields an upper bound on essential \(p\)-dimension in terms of the dimension and the \(p\)-adic valuation of the Chern numbers of a variety with a generically free action.

In \S\ref{sect:p-adic}, we briefly explain how the result of \S\ref{sect:conf} can be used to improve the applications of Theorem~\ref{thm_main} discussed above.

In Appendix~\ref{sect:rep_p}, we gather a few technical lemmas on representations of finite \(p\)-groups.

\subsection*{Acknowledgements}
The author is grateful to Zinovy Reichstein for his interest in this work and discussions about it.

\subsection*{Use of AI tools}
The author has written the entirety of the text of the paper, and checked every argument in it.

A large language model has been used to discover the application of the main result to Minkowski--Schur-type bounds; to produce the counterexample showing that the hypothesis on \(G[2]\) is necessary in (\ref{th:conf_main}); to find in the literature an explicit example of a group which does not satisfy the condition (GF); to search the literature in order to remove known results from an earlier version of the paper; to help verify the correctness of the statements and proofs, and to help fix the issues revealed in that way or simplify the arguments; and to polish the paper by improving the English and removing typos.

\section{\texorpdfstring{\(p\)}{p}-versality and weak \texorpdfstring{\(p\)}{p}-versality}
\label{sect:versality}

In this section and the next, we fix notation and establish some basic facts about \(p\)-versality and essential \(p\)-dimension.
We believe that most results of these two sections are standard (see e.g.\ \cite{Merkurjev-essential_contemp,Merkurjev-Bourbaki, Duncan-Reichstein-versal}), but we could not find them in the literature stated in the form needed here, particularly (\ref{prop:ed_smooth_proj}).

In addition, some care is taken here in order to include the case of varieties which are not geometrically connected:
they cannot be avoided in the formulation of (\ref{prop:ed_smooth_proj}), and are needed for the form of our main result in (\ref{prop:fpt}).

\subsection*{Conventions}
Throughout the paper, we fix a prime number \(p\), a field \(k\), and denote by \(\overline{k}\) an algebraic closure of \(k\).
When \(k\) has characteristic different from \(p\), we fix for each \(a \in \Nn\) a primitive \(p^a\)-th root of unity \(\zeta_{p^a} \in \overline{k}\).
By writing \(\zeta_{p^a} \in k\), we implicitly assume that \(k\) has characteristic different from \(p\).

When \(S\) and \(X\) are \(k\)-schemes, we denote by \(X_S\) the \(S\)-scheme \(X \times_k S\).
When \(S=\Spec L\) for \(L/k\) a field extension, we write \(X_L\) instead of \(X_S\).

By a \(k\)-group we will mean an affine group scheme of finite type over \(k\), and by a subgroup of a \(k\)-group we will mean a closed subgroup scheme over \(k\).

A \(k\)-variety will mean a quasiprojective \(k\)-scheme, not necessarily reduced or irreducible.
The index \(\ind X\) of a \(k\)-variety \(X\) is the gcd of the degrees of its closed points; we set \(\ind \varnothing=0\).
The residue field of a point \(x\in X\) will be denoted by \(\kappa(x)\).\\

Throughout \S\ref{sect:versality}, we fix a \(k\)-group \(G\).

\subsection*{Torsors}

\begin{para}
	Let \(F\) be a field containing \(k\).
	By a \(G\)-torsor over \(F\), we will mean an fppf \(G_F\)-torsor over \(\Spec F\).
\end{para}

\begin{para}
	\label{p:torsor_spread}
	If \(T\) is a \(G\)-torsor over a finitely generated field extension \(F/k\), then there exists an integral \(k\)-variety \(Y\) with an isomorphism of \(k\)-algebras \(F\simeq k(Y)\), a \(k\)-variety \(X\), and a \(G\)-torsor \(X \to Y\) whose generic fiber is \(T\).

	Indeed, write \(F=k(Y)\) for some integral affine \(k\)-variety \(Y\).
	By \cite[(8.8.2)]{ega-4-3} there exists a \(k\)-scheme \(X\), together with morphisms \(a\colon G \times_k X \to X\) and \(\pi \colon X\to Y\), such that the generic fiber of \(\pi\) is the \(G\)-torsor \(T\).
	After shrinking \(X\) and \(Y\), we may assume that \(a\) is a \(G\)-action and \(\pi\) a \(G\)-torsor, by \cite[(8.10.5), (11.2.6)]{ega-4-3}.
	The fact that \(G\) is affine and that \(X \to Y\) is a \(G\)-torsor implies by descent that \(X\) is an affine \(k\)-variety.
\end{para}

\begin{para}
	\label{p:torsor_pb}
	Let \(X',Y'\) be \(k\)-schemes, and \(X \to X', Y \to Y'\) be \(G\)-torsors.
	Let \(Y \to X\) be a \(G\)-equivariant morphism.
	Then the induced morphism \(Y \to X \times_{X'} Y'\) is an isomorphism, being a \(G\)-equivariant morphism between \(G\)-torsors over \(Y'\).
\end{para}

\begin{para}
	\label{p:quotient_finite}
	Let \(G\) be a finite \(k\)-group, and \(X\) a \(k\)-variety with a free \(G\)-action.
	Then there exists a \(k\)-variety \(Y\), and an fppf \(G\)-torsor \(X \to Y\) compatible with the \(G\)-action on \(X\).
	The existence of the \(G\)-torsor with \(Y\) a \(k\)-scheme follows from \stacks{07S7}, while the quasiprojectivity of \(Y\) over \(k\) is a consequence of \stacks{0BD4} and \stacks{036O}.
\end{para}

\subsection*{\texorpdfstring{\(p\)}{p}-versality}
\begin{para}
	\label{p:twist}
	Let \(T\) be a \(G\)-torsor over a field \(K\supset k\).
	When \(X\) is a \(k\)-variety, we set \(\tw{X}{T}=(X \times_k T)/G\);
	as explained in \cite[Proposition~2.12]{Forence-cyclic}, this is a \(K\)-variety, by our standing quasiprojectivity assumption.
\end{para}

\begin{para}
	\label{p:twist_points}
	Let \(T\) be a \(G\)-torsor over a field \(K\supset k\), and \(X\) a \(k\)-variety.
	Let \(M\) be a \(K\)-scheme (with the trivial \(G\)-action).
	Then an \(M\)-point of the \(K\)-variety \(\tw{X}{T}\) amounts precisely to a \(G\)-equivariant morphism \(T \times_K M \to X\).
\end{para}

\begin{para}
	\label{p:twist_split}
	Since \(G\)-torsors over an algebraically closed field extension of \(k\) are trivial, it follows that \((\tw{X}{T})_{\overline{K}} \simeq X_{\overline{K}}\), for \(\overline{K}\) any algebraic closure of \(K\).
	In particular the \(K\)-variety \(\tw{X}{T}\) is projective \stacks{0BDG}, resp.\ smooth \stacks{02UQ}, if the \(k\)-variety \(X\) is so.
\end{para}

\begin{definition}[{\cite{Duncan-Reichstein-versal}}]
	A \emph{twisting pair} \((T,K)\) for the group \(G\) is the data of an infinite field \(K\supset k\) and a \(G\)-torsor \(T\) over \(K\).

	Let \(X\) be a \(k\)-variety with a \(G\)-action.
	The variety \(X\) is \emph{weakly \(p\)-versal}, resp.\ \emph{weakly versal}, if for every twisting pair \((T,K)\) for the group \(G\), the \(K\)-variety \(\tw{X}{T}\) has index prime to \(p\), resp.\ a rational point.

	The variety \(X\) is \emph{\(p\)-versal}, resp.\ \emph{versal}, if every \(G\)-invariant dense open subscheme of \(X\) is weakly \(p\)-versal, resp.\ weakly versal.
\end{definition}

\begin{para}
	\label{p:versal_open}
	Let \(X\) be a \(k\)-variety with a \(G\)-action, and \(U\) a \(G\)-invariant dense open subscheme of \(X\).
	Then \(X\) is \(p\)-versal, resp.\ versal, if and only if \(U\) is so.

	Indeed, every \(G\)-invariant dense open subscheme of \(U\) is also a \(G\)-invariant dense open subscheme of \(X\), and therefore \(U\) is \(p\)-versal, resp.\ versal, as soon as \(X\) is so.
	Conversely, for any \(G\)-invariant dense open subscheme \(W\) of \(X\), and any twisting pair \((T,K)\), the \(G\)-invariant open subscheme \(W\cap U\) of \(U\) is dense.
	Thus if \(U\) is \(p\)-versal, resp.\ versal, then the \(K\)-variety \(\tw{(W\cap U)}{T}\) has a closed point of degree prime to \(p\), resp.\ a rational point.
	The same is true for \(\tw{W}{T}\), which contains \(\tw{(W\cap U)}{T}\) as an open subscheme.
\end{para}

The following is a slight generalization of \cite[Theorem~8.3]{Duncan-Reichstein-versal}, which assumed \(X\) geometrically connected and \(G\) smooth.
\begin{proposition}
	\label{prop:versal_weak}
	Let \(X\) be a smooth \(k\)-variety with a \(G\)-action.
	Then \(X\) is \(p\)-versal if and only if it is weakly \(p\)-versal.
\end{proposition}
\begin{proof}
	Let \(U\) be a \(G\)-invariant dense open subscheme of \(X\).
	Then for any twisting pair \((T,K)\), the \(K\)-variety \(\tw{U}{T}\) is a dense open subscheme of \(\tw{X}{T}\).
	Since the \(K\)-variety \(\tw{X}{T}\) is smooth by (\ref{p:twist_split}), it follows from \cite[Proposition~6.8]{Index} that the index of \(\tw{U}{T}\) equals that of \(\tw{X}{T}\), hence is prime to \(p\) if \(X\) is weakly \(p\)-versal.
	The converse is clear.
\end{proof}

\begin{para}
	\label{p:versal_torsor}
	A \(G\)-torsor \(T\) over a field \(F \supset k\) is called \emph{versal} if there exists a \(k\)-variety \(V\) with a versal \(G\)-action, and a \(G\)-torsor \(V \to W\) over an integral \(k\)-variety \(W\) with \(k(W)=F\) and such that \(T\) is the generic fiber of \(V \to W\).
\end{para}

\begin{para}
	\label{p:rep_versal}
	Note that a \(G\)-representation \(V\) is versal, hence \(p\)-versal for any \(p\).
	Indeed for any twisting pair \((T,K)\), and any \(G\)-invariant dense open subscheme \(U\) of \(V\), the \(K\)-variety \(\tw{U}{T}\) is an open dense subscheme of \(\tw{V}{T}\).
	Since the latter is an affine space over \(K\) (see e.g.\ \cite[Lemma~3.1]{Duncan-Reichstein-versal}), the former has \(K\)-points since \(K\) is infinite.
\end{para}

\begin{para}
	\label{p:generic_exists}
	It is proved in \cite[Remark~1.4]{Totaro-CHBG} that there always exists a \(G\)-representation \(V\) containing a \(G\)-invariant dense open subscheme \(V^{\circ}\) with \(V^{\circ}/G\) a \(k\)-variety and \(V^{\circ} \to V^{\circ}/G\) a \(G\)-torsor.
	Then the \(k\)-variety \(V^{\circ}\) is geometrically integral.
	Moreover, replacing \(V\) with \(V\oplus 1\) (where \(1\) denotes the trivial \(G\)-representation) and \(V^{\circ}\) with \(V^{\circ} \times_k \mathbb{A}^1\), we may assume that \(\dim (V^{\circ}/G)>0\), and in particular that the field \(k(V^{\circ}/G)\) is infinite.

	Setting \(F_0=k(V^{\circ}/G)\) and letting \(T_0\) be the generic fiber of \(V^{\circ} \to V^{\circ}/G\), we thus obtain a twisting pair \((T_0,F_0)\) such that the \(k\)-scheme \(T_0\) is geometrically integral.

	In addition \(V\) is versal by (\ref{p:rep_versal}), hence so is \(V^{\circ}\) by (\ref{p:versal_open}), and thus the \(G\)-torsor \(T_0\) over \(F_0\) is versal.
\end{para}

The following lemma is presumably standard (for instance when \(G\) acts freely on \(X\), it follows from \cite[Proposition~2.8]{Merkurjev-essential_contemp}).
\begin{lemma}
	\label{lemm:index_generic_twist}
	Let \(X\) be a \(k\)-variety with a \(G\)-action, and let \((T,F)\) be a twisting pair.
	Let \(T_0\) be a versal \(G\)-torsor over a field \(F_0\supset k\).
	Then the index of the \(F\)-variety \(\tw{X}{T}\) divides the index of the \(F_0\)-variety \(\tw{X}{T_0}\), and if \((\tw{X}{T_0})(F_0)\ne \varnothing\) then \((\tw{X}{T})(F)\ne \varnothing\).
\end{lemma}
\begin{proof}
	Let \(V \to W\) be as in (\ref{p:versal_torsor}), so that \(F_0=k(W)\).
	Let \(z\) be a closed point of degree \(d\) of the \(F_0\)-variety \(\tw{X}{T_0}\).
	This point spreads out to a finite locally free morphism \(D \to W'\) of degree \(d\), for some dense open subscheme \(W'\) of \(W\).
	Replacing \(W\) with \(W'\) and \(V\) with \(V \times_W W'\) (which remains versal by (\ref{p:versal_open})), we may assume that \(D \to W\) is finite locally free of degree \(d\).

	Set \(C = D \times_W V\), with \(G\) acting on \(C\) through \(V\).
	Then \(C \to V\) is finite locally free of degree \(d\);
	in particular the \(k\)-scheme \(C\) is quasiprojective.
	Note that
	\[
		C \times_W \Spec F_0=(D \times_W \Spec F_0) \times_{F_0} (V\times_W \Spec F_0)  =z \times_{F_0} T_0.
	\]
	By (\ref{p:twist_points}) the point \(z\) of \(\tw{X}{T_0}\) corresponds to a \(G\)-equivariant morphism \(C \times_W \Spec F_0=z \times_{F_0} T_0 \to X\).
	Shrinking \(W\), the latter extends to a \(G\)-equivariant morphism \(C \to X\) (see \cite[(8.8.2), (8.10.5)]{ega-4-3}).

	As \(F\) is infinite and \(V\) is versal, the \(F\)-variety \(\tw{V}{T}\) has a rational point.
	Since \(\tw{C}{T} \to \tw{V}{T}\) is finite locally free of degree \(d\), it follows that the \(F\)-variety \(\tw{C}{T}\) has an effective zero-cycle of degree \(d\).
	The existence of the morphism \(\tw{C}{T} \to \tw{X}{T}\) over \(F\) implies that the \(F\)-variety \(\tw{X}{T}\) has an effective zero-cycle of degree \(d\).
	The case \(d=1\) gives the statement about rational points.
	We also obtain that \(\ind(\tw{X}{T})\) divides the degree of every closed point of the \(F_0\)-variety \(\tw{X}{T_0}\), hence its index \(\ind(\tw{X}{T_0})\).
\end{proof}

\subsection*{The \texorpdfstring{\(G\)}{G}-index}

\begin{definition}
	\label{p:G_index}
	When \(X\) is a \(k\)-variety with a \(G\)-action, it follows from (\ref{lemm:index_generic_twist}) that the index of the \(F_0\)-variety \(\tw{X}{T_0}\) does not depend on the versal \(G\)-torsor \(T_0\) over a field \(F_0 \supset k\), provided that the field \(F_0\) is infinite.
	Since such a \(G\)-torsor exists by (\ref{p:generic_exists}), this allows us to define an integer \(\ind_G(X)\) as the index of the \(F_0\)-variety \(\tw{X}{T_0}\).
\end{definition}

\begin{para}
	\label{p:trans_versal}
	The \(k\)-variety \(X\) is weakly \(p\)-versal if and only if \(\ind_G(X)\) is prime to \(p\).
	Indeed \(\ind_G(X)\) is divisible by the index of \(\tw{X}{T}\) for every twisting pair \((T,K)\), and it is realized by the pair \((T_0,F_0)\) of (\ref{p:generic_exists}).
\end{para}

\begin{para}
	\label{p:funct_G_ind}
	If \(X \to Y\) is a \(G\)-equivariant morphism between \(k\)-varieties, then \(\ind_G(Y)\) divides \(\ind_G(X)\).
\end{para}

When the \(k\)-variety \(X\) is projective, we consider the degree map from the equivariant Chow group of \cite{EG-Equ}
\[
	\deg\colon \CH_G(X) \to \CH(X) \to \CH(\Spec k) = \Zz.
\]

\begin{proposition}
	\label{prop:versal_CH_G}
	Let \(X\) be a projective \(k\)-variety with a \(G\)-action.
	Then the image of \(\deg\colon\CH_G(X) \to \Zz\) is \(\ind_G(X) \cdot \Zz\).
\end{proposition}

\begin{proof}
	Note that the map \(\deg\) factors as \(\CH_G(X) \to \CH_G(\Spec k) \to \CH(\Spec k) =\Zz\).
	Let \((T_0,F_0)\) and \(V,V^{\circ}\) be as in (\ref{p:generic_exists}).
	Consider the diagram
	\[
		\begin{tikzcd}
			\CH_G(X) \ar[r] \ar[d]& \CH_G(X \times_k V^{\circ}) \ar[equal]{r} \ar[d] & \CH((X \times_k V^{\circ})/G) \ar[r]\ar[d]& \CH(\tw{X}{T_0})\ar[d]\\
			\CH_G(\Spec k) \ar[r] & \CH_G(V^{\circ}) \ar[equal]{r} & \CH(V^{\circ}/G) \ar[r]& \CH(\Spec F_0)
		\end{tikzcd}
	\]
	Here the vertical morphisms are proper pushforwards, the horizontal ones are flat pullbacks, and the equalities are the natural identifications of \cite[Proposition~8]{EG-Equ}.
	The diagram commutes because proper pushforwards commute with flat pullbacks.
	In addition, the horizontal maps preserve the codimension gradings of the groups in the diagram.

	The upper-left horizontal morphism decomposes as \(\CH_G(X) \to \CH_G(X \times_k V) \to \CH_G(X \times_k V^{\circ})\), and so is surjective by homotopy invariance and the localization sequence.
	The upper-right horizontal morphism is also surjective by the continuity property of Chow groups (see e.g.\ \cite[Proposition~52.9]{EKM}), as \(\tw{X}{T_0}\) is the limit of the open subschemes \((X \times_k W)/G\) of \((X \times_k V^{\circ})/G\), where \(W\) runs over the \(G\)-invariant nonempty open subschemes of \(V^{\circ}\) (note that \(\tw{X}{T_0}\) is the generic fiber of the morphism \((X \times_k V^{\circ})/G \to V^{\circ}/G\)).
	Therefore the image of the total composite in the diagram coincides with the image of the right vertical map, namely \(\ind_G(X) \cdot \Zz\).

	Now, the lower horizontal composite is a morphism of graded groups with target \(\Zz\) concentrated in degree zero, which sends \(1\) to \(1\) in degree zero.
	Since \(\CH^0_G(\Spec k)\) is generated by \(1\), there is only one such morphism, and it coincides with \(\CH_G(\Spec k) \to \CH(\Spec k)=\Zz\).
	Thus the total composite in the diagram is the map \(\deg\) of the statement.
	The statement follows.
\end{proof}

\begin{para}
	\label{p:versal_CH_G}
	By (\ref{p:trans_versal}) and (\ref{prop:versal_CH_G}), a projective \(k\)-variety \(X\) is weakly \(p\)-versal if and only if the morphism \(\deg\colon\CH_G(X) \to \Fp\) is nonzero.
\end{para}

\subsection*{\texorpdfstring{\(G\)}{G}-connectedness}
\begin{definition}
	A \(k\)-variety \(Y\) with a \(G\)-action is called \emph{\(G\)-connected} if it is nonempty and cannot be written as a nontrivial disjoint union of two open \(G\)-invariant subschemes.
\end{definition}

\begin{para}
	\label{p:G_connected_decomp}
	Any \(k\)-variety with a \(G\)-action has finitely many connected components, and thus decomposes \(G\)-equivariantly as a disjoint union of \(G\)-connected open subschemes.
\end{para}

\begin{para}
	\label{p:G_connected_transitive}
	If \(G\) is a finite constant \(k\)-group acting on a \(k\)-variety \(X\), then \(X\) is \(G\)-connected if and only if the group \(G(k)\) acts transitively on the set of connected components of \(X\).

	Indeed if \(Y\) is a connected component of \(X\), then \(\bigcup_{g\in G(k)} gY\) is a \(G\)-invariant open and closed subscheme of \(X\).
	Being nonempty, it can only be \(X\) itself when \(X\) is \(G\)-connected.
	Conversely, any nonempty \(G\)-invariant open and closed subscheme of \(X\) contains a connected component of \(X\), hence the whole \(X\) if \(G(k)\) acts transitively on the set of connected components of \(X\).
\end{para}

\begin{para}
	\label{p:G_connected_equidim}
	A \(G\)-connected \(k\)-variety which is locally equidimensional (for instance smooth) is equidimensional.
	Indeed such \(X\) decomposes as the disjoint union of varieties \(X_n\) of pure dimension \(n\), for \(n \in \Nn\).
	The morphism of \(k\)-schemes
	\[
		G \times_k X \to G \times_k X, \quad (g,x) \mapsto (g,gx)
	\]
	is an isomorphism over \(G\), hence preserves the local dimensions of the fibers of the projections to \(G\).
	Thus it maps \(G \times_k X_n\) to itself.
	Therefore each \(X_n\) is a \(G\)-invariant open and closed subscheme of \(X\).
	If \(X\) is \(G\)-connected, we must have \(X=X_n\) for some \(n\), and so \(X\) has pure dimension \(n\).
\end{para}

\begin{lemma}
	\label{lemm:geom_comp_div}
	Let \(X\) be a projective \(k\)-variety with a \(G\)-action.
	If \(X\) is \(G\)-connected, then \(\ind_G(X)\) is divisible by the number of geometric connected components of \(X\).
\end{lemma}
\begin{proof}
	Let \(d\) be the number of geometric connected components of \(X\).
	Let \(R=\Spec H^0(X,\Oc_X)\), a finite \(k\)-scheme.
	Then \(R\) inherits a \(G\)-action, is \(G\)-connected, and its number of geometric connected components is \(d\).
	We have \(\ind_G(R) \mid \ind_G(X)\) by (\ref{p:funct_G_ind}).

	Let now \((T_0, F_0)\) be as in (\ref{p:generic_exists}), and recall that the \(k\)-scheme \(T_0\) is geometrically connected.
	By \stacks{0385} the projection \(R \times_k T_0 \to R\) induces a bijection between the connected components of \(R\) and those of \(R \times_k T_0\).
	Moreover this projection is \(G\)-equivariant.
	We deduce that \(R \times_k T_0\) is \(G\)-connected.
	Since \(G\)-invariant open and closed subschemes of \(R \times_k T_0\) correspond bijectively to open and closed subschemes of \(\tw{R}{T_0}=(R \times_k T_0)/G\), we conclude that the scheme \(\tw{R}{T_0}\) is connected.
	Being of dimension \(0\), the scheme \(\tw{R}{T_0}\) is therefore irreducible.

	Since \((\tw{R}{T_0})_{\overline{F}} \simeq R_{\overline{F}}\) for an algebraic closure \(\overline{F}\) of \(F_0\) (see (\ref{p:twist_split})), the \(F_0\)-variety \(\tw{R}{T_0}\) has \(d\) geometric connected components.
	The same holds for the \(F_0\)-variety \((\tw{R}{T_0})_{\red}\).
	But \((\tw{R}{T_0})_{\red}=\Spec E\), for \(E/F_0\) a field extension of finite degree.
	Its degree \([E:F_0]\) is divisible by \([E:F_0]_{\mathrm{sep}}=d\).
	The index of the \(F_0\)-variety \(\tw{R}{T_0}\) equals that of \((\tw{R}{T_0})_{\red}\), which equals \([E:F_0]\).
	The statement follows.
\end{proof}

\section{Generic freeness and essential \texorpdfstring{\(p\)}{p}-dimension}
\label{sect:essential}

In this section we discuss a basic relation between faithfulness and generic freeness of a group action on a variety (already observed e.g.\ in \cite[Example~5.5]{Garibaldi-Merkurjev-Serre}), and recall the definition of essential \(p\)-dimension introduced in \cite{Reichstein-Youssin}.

Here \(G\) still denotes a \(k\)-group.

\subsection*{Generic freeness}
\begin{para}
	A \(G\)-action on a \(k\)-scheme \(X\) is \emph{faithful} if \(G \to \Aut_k(X)\) is a monomorphism.
\end{para}

\begin{para}
	A \(G\)-action on a \(k\)-variety \(X\) is called \emph{free} if \(G \times_k X \to X \times_k X\) is a monomorphism, and \emph{generically free} if \(X\) contains a \(G\)-invariant dense open subscheme on which \(G\) acts freely.
\end{para}

\begin{para}
	We say that the group \(G\) satisfies \(\GF\) if every faithful \(G\)-action on a geometrically integral \(k\)-variety is generically free.
\end{para}

\begin{para}
	\label{p:finite_genfree}
	The \(k\)-group \(G\) satisfies \(\GF\) when it is finite étale.
	Indeed, let us assume that \(G\) acts faithfully on a geometrically integral \(k\)-variety \(X\), and show that the action is generically free.

	Let \(L/k\) be a finite Galois extension such that the \(L\)-group \(G_L\) is finite constant.
	The complements of the fixed closed subschemes \((X_L)^g\), for \(g\in G(L)\smallsetminus \{1\}\), are nonempty (because the \(G_L\)-action is faithful and \(X_L\) is reduced).
	Their intersection \(U\) is thus nonempty, because \(X_L\) is irreducible.
	This is a \(G_L\)-invariant open subscheme of \(X_L\), on which \(G_L\) acts freely.
	It is also stable under the action of \(\Gal(L/k)\), and so is of the form \(W_L\), for some open subscheme \(W\) of \(X\) over \(k\).
	The open subscheme \(W\) is \(G\)-invariant in \(X\) and the \(G\)-action on \(W\) is free, because this becomes true after extending scalars to \(L\).
\end{para}

\begin{remark}
	Other examples of groups satisfying \(\GF\) include groups of multiplicative type.
\end{remark}

\begin{remark}
	\label{rem:O_2_GF}
	Let \(G=\Gm \rtimes (\Zz/2)\), with the nontrivial element \(\sigma\) of \(\Zz/2\) acting on \(\Gm\) by \(t\mapsto t^{-1}\).
	Note that \(G\simeq \Oo(h)\), in any characteristic, where \(h\) is the hyperbolic plane.
	Then, as explained in \cite[Introduction]{LMMR-ext}, the \(k\)-group \(G\) does not satisfy the condition \(\GF\).
\end{remark}

Finite constant abelian groups satisfy a stronger condition than \(\GF\):
\begin{proposition}
	\label{lemm:GF_abelian}
	Let \(G\) be a finite constant abelian \(k\)-group acting faithfully on a reduced and locally irreducible (e.g.\ smooth) \(k\)-variety \(X\).
	If \(X\) is \(G\)-connected, then the \(G\)-action on \(X\) is generically free.
\end{proposition}
\begin{proof}
	Let \(g\in G(k)\).
	Assume that some connected component \(Y\) of \(X\) satisfies \(Y \subset X^g\).
	Let \(Y'\) be a connected component of \(X\).
	Then \(Y'=h Y\) for some \(h \in G(k)\), by (\ref{p:G_connected_transitive}).
	Let \(S\) be a \(k\)-scheme, and \(y'\in Y'(S)\).
	Then \(y=h^{-1}y' \in Y(S)\), and
	\[
		gy'=ghy=hgy=hy=y',
	\]
	showing that \(Y' \subset X^g\).
	We conclude that \(g\) acts trivially on \(X\), hence \(g=1\) by faithfulness.

	Now, any irreducible component \(Z\) of \(X\) is a connected component, and as we have just seen \(Z \not \subset X^g\) when \(g\in G(k) \smallsetminus \{1\}\).
	Since \(Z\) is irreducible and reduced, it follows that \(Z \not \subset \bigcup_{g\in G(k) \smallsetminus \{1\}} X^g\) set-theoretically.
	This shows that the open complement of \(\bigcup_{g\in G(k) \smallsetminus \{1\}} X^g\) is dense in \(X\).
	It is also \(G\)-invariant, with a free \(G\)-action.
\end{proof}

\subsection*{Essential \texorpdfstring{\(p\)}{p}-dimension}

\begin{para}
	Let \(F\) be a field containing \(k\).
	If \(F \subset L\) is a field extension, and \(T\) a \(G\)-torsor over \(F\), there is an induced \(G\)-torsor \(T_L\) over \(L\).

	Given a \(G\)-torsor \(T\) over \(F\), a subfield \(F_0 \subset F\) containing \(k\) is a \emph{field of definition of \(T\)} if there exist a \(G\)-torsor \(T_0\) over \(F_0\) and a \(G_F\)-equivariant isomorphism of \(F\)-schemes \((T_0)_F \simeq T\).

	The \emph{essential dimension} of a \(G\)-torsor \(T\) over a field \(F \supset k\) is defined as
	\[
		\ed(T) = \min \trdeg_k F_0,
	\]
	where \(F_0\subset F\) runs over the fields of definition of \(T\).
	Since by \cite[(8.8.2), (8.10.5)]{ega-4-3}, any such field \(F_0\) contains a field of definition of \(T\) which is finitely generated over \(k\), we may restrict to finitely generated extensions \(F_0/k\) while taking the above minimum.
\end{para}

\begin{para}
	The \emph{essential \(p\)-dimension} of a \(G\)-torsor \(T\) over a field \(F \supset k\) is defined as
	\[
		\ed_p(T) = \min \ed(T_L),
	\]
	where \(L/F\) runs over the field extensions of finite degree prime to \(p\).

	The \emph{essential \(p\)-dimension} of \(G\) is defined as
	\[
		\ed_p(G) = \sup \ed_p(T),
	\]
	where \(T\) runs over the \(G\)-torsors over the fields containing \(k\).
\end{para}

The following statement is well-known:
\begin{proposition}
	\label{prop:ed_versal}
	The number \(\ed_p(G) + \dim G\) is the least dimension of a \(k\)-variety with a \(p\)-versal and generically free \(G\)-action.
\end{proposition}

\begin{proof}
	Let \(X\) be a \(k\)-variety with a \(p\)-versal and generically free \(G\)-action.
	Let \(U\) be a \(G\)-invariant dense open subscheme of \(X\) where \(G\) acts freely.
	Shrinking \(U\), we may further assume that the fppf quotient \(U/G\) exists as a \(k\)-scheme, by \cite[Proposition~4.7]{Thomason-comparison}.
	Let \(T\) be a \(G\)-torsor over a field \(K \supset k\).
	If \(K\) is finite, then \(\ed_p(T)=0\).
	Otherwise \((T,K)\) is a twisting pair for \(G\).
	Since \(X\) is \(p\)-versal, the variety \(U\) is weakly \(p\)-versal.
	Thus the \(K\)-variety \(\tw{U}{T}\) has an \(E\)-point, for some field extension \(E/K\) of finite degree prime to \(p\).
	By (\ref{p:twist_points}) this gives a \(G\)-equivariant \(T_E \to U\).
	By (\ref{p:torsor_pb}), we deduce that \(T_E \simeq U \times_{U/G} \Spec E\).
	Thus there exists a point \(u\) of \(U/G\) and an inclusion \(k \subset \kappa(u) \subset E\) such that the torsor \(T_E\) over the field \(E\) is defined over the field \(\kappa(u)\), and so
	\[
		\ed_p(T) \le \trdeg_k \kappa(u) \le \dim U/G= \dim X -\dim G.
	\]
	Therefore \(\ed_p(G) \le \dim X -\dim G\).\\

	Conversely, let \((T_0, F_0)\) be as in (\ref{p:generic_exists}).
	Let \(L/F_0\) be a field extension of finite degree prime to \(p\), and \(F\supset k\) a field of definition of \((T_0)_L\) finitely generated over \(k\), such that \(\trdeg_k F=\ed_p(T_0)\).
	So there exists a \(G\)-torsor \(S\) over \(F\) such that \(S_L \simeq (T_0)_L\).
	Then by (\ref{p:torsor_spread}) there exists a \(k\)-variety \(B\) and a \(G\)-torsor \(Y \to B\) whose generic fiber is \(S \to \Spec F\).
	Note that the \(G\)-action on \(Y\) is free.

	Assume now that \(U\) is a \(G\)-invariant dense open subscheme of \(Y\).
	Then the morphism \(\Spec L \to \Spec F \to U/G\) lifts to a \(G\)-equivariant morphism \((T_0)_L\to U\), yielding by (\ref{p:twist_points}) an \(L\)-point of the \(F_0\)-variety \(\tw{U}{T_0}\).
	Therefore, by (\ref{p:trans_versal}), the \(k\)-variety \(U\) is weakly \(p\)-versal.
	We conclude that \(Y\) is \(p\)-versal.
	We have
	\[
		\dim Y = \dim B + \dim G = \trdeg_k F + \dim G = \ed_p(T_0) + \dim G \le \ed_p(G) + \dim G.
	\]
	The first part of the proof implies that \(\dim Y= \ed_p(G) + \dim G\), with the \(G\)-action on \(Y\) free and \(p\)-versal.
\end{proof}

\begin{para}
	\label{p:ed_X_G_smooth}
	Assume that the field \(k\) is perfect and that \(G\) is smooth over \(k\).
	We claim that in (\ref{prop:ed_versal})  we may additionally require the variety to be smooth over \(k\).

	Indeed assume that \(X\) is a \(k\)-variety with a generically free and \(p\)-versal \(G\)-action.
	Since \(G\) is smooth, the closed subscheme \(X_{\red}\) is \(G\)-invariant.
	It carries a generically free and \(p\)-versal \(G\)-action (for any twisting pair \((T,K)\) the morphism \(\tw{(X_{\red})}{T} \to \tw{X}{T}\) is a nilimmersion).
	Replacing \(X\) with \(X_{\red}\), we may thus assume that \(X\) is reduced.
	Let \(X^{\sm}\) be the smooth locus of \(X\), a \(G\)-invariant open subscheme which is dense because \(k\) is perfect.
	The \(k\)-variety \(X^{\sm}\) is then the required \(p\)-versal (by (\ref{p:versal_open})) and generically free variety.
\end{para}

\begin{para}
	\label{p:res_sing}
	If \(k\) has characteristic zero, by (\ref{p:ed_X_G_smooth}), (\ref{p:versal_open}), and equivariant resolution of singularities \cite[Corollary~3.6, Remark~3.3]{Reichstein-Youssin}, we may require the \(k\)-variety in (\ref{prop:ed_versal}) to be smooth and projective.
\end{para}

\begin{proposition}
	\label{prop:ed_smooth_proj}
	Assume that \(k\) has characteristic zero.
	Then \(\ed_p(G)+\dim G\) is the least dimension of a smooth projective \(k\)-variety \(X\) with a generically free \(G\)-action such that \(\deg\colon\CH_G(X) \to \Fp\) is nonzero.
\end{proposition}
\begin{proof}
	This follows by combining (\ref{p:versal_CH_G}), (\ref{prop:versal_weak}) and (\ref{prop:ed_versal}), in view of (\ref{p:res_sing}).
\end{proof}

\section{Chern numbers and \texorpdfstring{\(p\)}{p}-versality}
\label{sect:Chern}
We fix a \(k\)-group \(G\) in this section.
\begin{para}
	\label{p:Chern_equiv}
	Let \(X\) be a smooth projective \(k\)-variety.
	A Chern number of \(X\) is the degree, in \(\Zz\), of a polynomial with integral coefficients in the Chern classes of the tangent bundle \(\Tan_X\).
	Recall from \cite{EG-Equ} that Chern classes of equivariant vector bundles lift to the equivariant Chow ring.
	When \(G\) acts on \(X\), the tangent bundle \(\Tan_X\) is \(G\)-equivariant, so that the Chern numbers of \(X\) belong to the image of \(\deg\colon\CH_G(X) \to \Zz\).
\end{para}

\begin{proposition}
	\label{prop:Chern_versal}
	Let \(X\) be a smooth projective \(k\)-variety with a \(G\)-action.
	Then:
	\begin{enumerate}[label=(\roman*), ref=\roman*]
		\item
			\label{prop:Chern_versal:1}
			every Chern number of \(X\) is divisible by \(\ind_G(X)\),

		\item 
			\label{prop:Chern_versal:2}
			if a Chern number of \(X\) is prime to \(p\), then \(X\) is \(p\)-versal.
	\end{enumerate}
\end{proposition}
\begin{proof}
	\eqref{prop:Chern_versal:1}: This follows from (\ref{prop:versal_CH_G}) and (\ref{p:Chern_equiv}).

	\eqref{prop:Chern_versal:2}: The \(k\)-variety \(X\) is weakly \(p\)-versal by \eqref{prop:Chern_versal:1} and (\ref{p:trans_versal}), and therefore \(p\)-versal by (\ref{prop:versal_weak}).
\end{proof}

\begin{theorem}
	\label{th:main}
	Let \(X\) be a smooth projective \(k\)-variety with a generically free \(G\)-action.
	If a Chern number of \(X\) is prime to \(p\), then \(\ed_p(G) + \dim G \le \dim X\).
\end{theorem}
\begin{proof}
	This follows from (\ref{prop:Chern_versal}.\ref{prop:Chern_versal:2}) and (\ref{prop:ed_versal}).
\end{proof}

\begin{para}
	When \(G\) acts on a \(k\)-variety \(X\), the \(k\)-group \(H = \ker (G \to \Aut_k(X))\) is a closed normal subgroup of \(G\) (see e.g.\ \cite[VI\(_B\), 6.2.4 e), p.370]{SGA3-1}).
	The \(k\)-group \(Q=G/H\) then acts faithfully on \(X\).
\end{para}

\begin{theorem}
	\label{th:quotient}
	Let \(X\) be a smooth, projective, geometrically connected \(k\)-variety with a \(G\)-action.
	Assume that \(X\) has a Chern number prime to \(p\).
	Let \(Q=G/\ker (G \to \Aut_k(X))\).
	If \(Q\) satisfies \(\GF\), then \(\ed_p(Q) + \dim Q\le \dim X\).
\end{theorem}
\begin{proof}
	The \(k\)-variety \(X\) is geometrically integral, being smooth and geometrically connected.
	Thus, by \(\GF\), the \(Q\)-action on \(X\) is generically free, and the statement follows from (\ref{th:main}).
\end{proof}

\begin{proposition}
	\label{prop:fpt}
	Let \(G\) be a finite étale \(k\)-group such that \(G(\overline{k})\) is a \(p\)-group.
	Let \(X\) be a smooth projective \(k\)-variety with a \(G\)-action.
	Let \(K=H^0(X,\Oc_X)\).
	Assume that \(X\) is \(G\)-connected and has a Chern number prime to \(p\).

	Then \(X\) is connected, and \(K/k\) is a finite separable field extension of degree prime to \(p\).
	The \(K\)-group \(G_K\) acts on the \(K\)-variety \(X\).
	Set
	\[
		Q=G/\ker ( G \to \Aut_k(X)) \quad \text{ and } \quad Q'=G_K/\ker (G_K \to \Aut_K(X)).
	\]
	Then \(\ed_p(Q')\le \dim X\), and \(Q \subset R_{K/k}(Q') \subset \Aut_k(X)\).

	If the \(k\)-group \(G\) is finite constant, then so is \(Q\), and \(Q'=Q_K\).
\end{proposition}
\begin{proof}
	The geometric connected components of \(X\) are in bijection with those of \(\Spec K\).
	The \(k\)-variety \(\Spec K\) is geometrically reduced, being scheme-theoretically dominated by \(X\).
	Thus the common number of geometric connected components is \(d=\dim_kK\).
	Since a Chern number of \(X\) is prime to \(p\), it follows from (\ref{prop:Chern_versal}.\ref{prop:Chern_versal:1}) and (\ref{lemm:geom_comp_div}) that \(d\) is prime to \(p\).

	Let \(\Gamma=\Aut(\overline{k}/k)\).
	The group \(G(\overline{k})\rtimes \Gamma\) acts on the set \(\Cc\) of connected components of \(X_{\overline{k}}\).
	The \((G(\overline{k})\rtimes \Gamma)\)-stable subsets of \(\Cc\) correspond to the \(G\)-invariant open and closed subschemes of \(X\).
	Since \(X\) is \(G\)-connected, the \(G(\overline{k})\rtimes \Gamma\)-action on \(\Cc\) is transitive.
	As \(G(\overline{k})\) is normal in \(G(\overline{k}) \rtimes \Gamma\), the subgroup \(\Gamma\) permutes transitively the \(G(\overline{k})\)-orbits; they thus have the same cardinality, a power \(p^a\) of \(p\) since \(G(\overline{k})\) is a \(p\)-group.
	Since \(p^a\) divides \(|\Cc|=d\), and \(d\) is prime to \(p\), we have \(a=0\).
	Therefore \(G(\overline{k})\) acts trivially on \(\Cc\), and so \(\Gamma\) acts transitively on \(\Cc\).
	We have established that \(X\) is connected.

	The \(k\)-algebra \(K=H^0(X,\Oc_X)\) is thus a separable field extension, as its spectrum is connected and geometrically reduced.
	The group \(G\) acts on \(\Spec K\).
	The set of geometric connected components of \(\Spec K\) is \(G(\overline{k})\rtimes \Gamma\)-equivariantly in bijection with \(\Cc\).
	In particular they carry the trivial \(G(\overline{k})\)-action.
	Since \(\Spec K\) is an étale \(k\)-scheme, this implies that \(G\) acts trivially on \(\Spec K\).

	The morphism \(X \to \Spec k\) factors \(G\)-equivariantly through \(\Spec K\).
	The \(K\)-variety \(X\) is geometrically connected, smooth, and projective, and it carries a \(G_K\)-action.
	Its dimension equals the dimension of the \(k\)-variety \(X\).
	It has a Chern number prime to \(p\), since the Chern numbers of the \(k\)-variety \(X\) equal those of the \(K\)-variety \(X\) multiplied by the integer \(d=[K:k]\), which is prime to \(p\) (note that \(\Omega_{K/k}=0\), and so \(\Omega_{X/k}=\Omega_{X/K}\)).

	We may thus apply (\ref{th:quotient}) to the \(K\)-variety \(X\), as the \(K\)-group \(Q'\) is finite étale, hence satisfies \(\GF\) by (\ref{p:finite_genfree}).
	We obtain that \(\ed_p(Q') \le \dim X\).
	The action of the \(k\)-group \(G\) on \(X\) factors as 
	\[
		G \to R_{K/k}(Q') \subset R_{K/k}(\Aut_K(X)) \subset \Aut_k(X),
	\]
	so that \(Q \subset R_{K/k}(Q')\).
	Finally, assume that the \(k\)-group \(G\) is finite constant.
	Then so is its quotient \(Q\).
	Since \(G\) acts trivially on \(K\), the elements of \(G(k)\) act on \(X\) by \(K\)-automorphisms.
	This implies that \(Q'=Q_K\).
\end{proof}

\begin{remark}
	A connected \(k\)-variety with a \(G\)-action is \(G\)-connected.
	The converse holds in the situation of (\ref{prop:fpt}).
\end{remark}

\begin{corollary}
	\label{cor:quotient}
	Assume that the characteristic of \(k\) differs from \(p\), and that the \(k\)-group \(G\) is a finite constant \(p\)-group.
	Let \(X\) be a smooth projective \(k\)-variety with a \(G\)-action.
	Assume that \(X\) is \(G\)-connected and has a Chern number prime to \(p\).

	Then the \(G(k)\)-action on \(X\) factors through a faithful action of a subgroup of \(\GL_n(k(\zeta_p))\), where \(n=\dim X\).
\end{corollary}
\begin{proof}
	We apply (\ref{prop:fpt}), and use its notation.
	We obtain that \(G\) acts on \(X\) through a quotient \(Q\), necessarily a finite constant \(p\)-group as well, such that \(\ed_p(Q_K) \le n\).
	Since \([K:k]\) is prime to \(p\), we have \(\ed_p(Q_K)=\ed_p(Q)\ge \ed_p(Q_{k(\zeta_p)})\) by \cite[Proposition~1.5]{Merkurjev-essential_contemp}.
	By \cite[Theorem~4.1]{KM-Ess}, the group \(Q(k)\) has a faithful representation over \(k(\zeta_p)\) of dimension equal to \(\ed_p(Q_{k(\zeta_p)}) \le n\), hence it is a subgroup of \(\GL_n(k(\zeta_p))\).
\end{proof}

\begin{remark}
	\label{rem:roots_are_needed}
	Adjoining a primitive \(p\)-th root of unity in (\ref{cor:quotient}) cannot be avoided:
	the \(\Qq\)-variety \(X=\Pp^1\) has the Chern number \(\deg c_1(\Tan_X)=2\) prime to \(3\), and an automorphism of order \(3\) given by \([x:y] \mapsto [-y:x-y]\).
	Yet \(\Zz/3\) is not a subgroup of \(\GL_1(\Qq)=\Qq^\times\), precisely because \(\zeta_3 \not \in \Qq\).
\end{remark}

\section{Abelian actions}
\label{sect:noncommutative}
In this section \(k\) is a field of characteristic different from \(p\).

\begin{lemma}
	\label{lemm:faithful_abelian}
	Let \(n_1,\dotsc,n_r \in \Nn\smallsetminus \{0\}\), and assume that \(G=(\Zz/p^{n_1}) \times\dotsb \times (\Zz/p^{n_r})\).
	Then the least dimension of a faithful \(G\)-representation over \(k\) is
	\[
		[k(\zeta_{p^{n_1}}):k] + \dotsb + [k(\zeta_{p^{n_r}}):k].
	\]
\end{lemma}
\begin{proof}
	The \(k\)-vector space \(k(\zeta_{p^{n_1}}) \times \dotsb \times k(\zeta_{p^{n_r}})\) carries a faithful \(G\)-action given by the formula
	\[
		(g_1,\dotsc,g_r) \cdot (x_1,\dotsc,x_r)= ( (\zeta_{p^{n_1}})^{g_1}x_1, \dotsc, (\zeta_{p^{n_r}})^{g_r}x_r).
	\]

	Conversely assume that \(V\) is a faithful \(G\)-representation over \(k\).
	Write \(V= W_1 \oplus \dotsb \oplus W_s\), where each \(W_i\) is an irreducible \(G\)-representation.
	The image of \(k[G]\) in \(\End_k(W_i)\) is a field \(K_i\) (as \(G\) is abelian), so the image of \(G\) is a finite subgroup of \(K_i^{\times}\), hence a cyclic group.
	Thus the \(G\)-representation \(W_i\) factors through a faithful representation of a quotient of \(G\) isomorphic to \(\Zz/p^{a_i}\) for some \(a_i\ge 0\).
	The field extension \(K_i/k\) is generated by the roots of unity of order dividing \(p^{a_i}\), hence \(K_i=k(\zeta_{p^{a_i}})\).
	Since \(W_i\) is irreducible as a \(k[G]\)-module, we have \(\dim_{K_i} W_i=1\), and so \(\dim_k W_i = [k(\zeta_{p^{a_i}}):k]\).

	As \(V\) is a faithful \(G\)-representation, the group \(G\) is a subgroup of the product of its images in \(\Aut_k(W_1),\dotsc, \Aut_k(W_s)\), namely of \((\Zz/p^{a_1}) \times \dotsb \times (\Zz/p^{a_s})\).
	Therefore \(s\ge r\), and upon reordering the \(n_i\)'s and \(a_i\)'s, we have \(n_i \le a_i\) for each \(i=1,\dotsc,r\) (to see this, observe that for an abelian \(p\)-group \(A = \Zz/p^{m_1} \times\dotsb \times \Zz/p^{m_t}\), the number of indices \(i\) such that \(m_i > j\) is the dimension over \(\Fp\) of the \(p\)-torsion subgroup of \(p^jA\)).
	Therefore
	\[
		\dim_k V= \sum_{i=1}^s \dim_k W_i = \sum_{i=1}^s [k(\zeta_{p^{a_i}}):k] \ge \sum_{i=1}^r [k(\zeta_{p^{n_i}}):k].\qedhere
	\]
\end{proof}

\begin{proposition}
	\label{prop:d_p_constant}
	Assume that \(G=(\Zz/p^{n_1}) \times\dotsb \times (\Zz/p^{n_r})\), viewed as a finite constant \(k\)-group, with \(n_1,\dotsc,n_r \in \Nn\smallsetminus\{0\}\).
	Let \(X\) be a smooth projective \(k\)-variety with a faithful action of \(G\).
	Assume that \(X\) is \(G\)-connected, and has a Chern number prime to \(p\).
	Then
	\[
		[k(\zeta_{p^{n_1}}):k(\zeta_p)] + \dotsb + [k(\zeta_{p^{n_r}}):k(\zeta_p)] \le \dim X.
	\]
\end{proposition}
\begin{proof}
	This follows from (\ref{cor:quotient}) and (\ref{lemm:faithful_abelian}) applied over \(k(\zeta_p)\).
\end{proof}

\begin{corollary}
	\label{cor:exponent}
	Let \(X\) be a smooth projective \(k\)-variety, and \(g\) an automorphism of \(X\) whose order is \(p^s\), for some \(s\in \Nn\smallsetminus\{0\}\).
	Assume that \(g\) permutes transitively the connected components of \(X\) (this is so, for instance, when \(X\) is connected), and that \(X\) has a Chern number prime to \(p\).
	Then \([k(\zeta_{p^s}):k(\zeta_p)] \le \dim X\).
\end{corollary}
\begin{proof}
	Let \(G\) be the finite constant \(k\)-group \(\Zz/p^s\), acting on \(X\) via \(g\).
	Then \(X\) is \(G\)-connected by (\ref{p:G_connected_transitive}), and the \(G\)-action on \(X\) is faithful.
	The statement thus follows from (\ref{prop:d_p_constant}).
\end{proof}

\section{Minkowski--Schur order bounds}
\label{sect:MS}
In this section \(k\) is a field of characteristic different from \(p\).\\

Let us recall a bound on the order of finite \(p\)-subgroups of \(\GL_n(F)\) for a field \(F\), due to Minkowski, Schur and Serre:
\begin{para}
	\label{p:Minkowski_bound}
	Let \(F\) be a field.
	We set, in \(\Nn \cup \{\infty\}\),
	\[
		M_p(n,F) = \sup \{v_p(|A|) \text{ where \(A\) runs over the finite \(p\)-subgroups of \(\GL_n(F)\)}\}.
	\]
\end{para}

\begin{para}
	\label{p:m_F}
	Let \(F\) be a field such that \(k \subset F \subset \overline{k}\) and \(\zeta_p\in F\).
	We set
	\[
		m_F=\sup \{a \ge 1 \text{ such that } \zeta_{p^a} \in F\} \in \Nn \cup \{\infty\}.
	\]
\end{para}

\begin{proposition}
	\label{prop:Minkowski_bound}
	Let \(n \in \Nn \smallsetminus \{0\}\).
	Let \(F\) be a field such that \(k \subset F \subset \overline{k}\) and \(\zeta_p\in F\).
	Assume that \(p\) is odd, or that \(\zeta_4 \in F\) when \(p=2\).
	Let us use the notation of (\ref{p:m_F}), with \(m=m_F\).
	Then (the right-hand side being understood as \(\infty\) when \(m=\infty\))
	\[
		M_p(n,F)=mn + v_p(n!) = mn + \Big \lfloor \frac{n}{p} \Big \rfloor + \Big \lfloor \frac{n}{p^2} \Big \rfloor + \dotsb.
	\]
\end{proposition}
\begin{proof}
	If \(m=\infty\), then for each \(a \in \Nn \smallsetminus \{0\}\), the subgroup of diagonal matrices \((\mu_{p^a})^n(F)\) of \(\GL_n(F)\) has order \(p^{an}\).
	Thus \(M_p(n,F)=\infty\), as claimed.

	Suppose now that \(m<\infty\).
	Since \(\zeta_{p^m} \in F\), the subgroup of diagonal matrices \((\mu_{p^m})^n(F)\) in \(\GL_n(F)\) has order \(p^{mn}\).
	Let \(P_n\) be a \(p\)-Sylow subgroup of the symmetric group \(S_n\), viewed as a subgroup of permutation matrices in \(\GL_n(F)\).
	Then \(P_n\) normalizes the subgroup \((\mu_{p^m})^n(F)\) in \(\GL_n(F)\), and \(Q_n=(\mu_{p^m})^n(F) \rtimes P_n\) is a \(p\)-subgroup of \(\GL_n(F)\) whose order satisfies \(v_p(|Q_n|)= mn + v_p(n!)\). This shows that \(M_p(n,F) \ge mn + v_p(n!)\).

	The other inequality is proved in \cite[\S5.4, Theorem 5]{Serre-Gk}, where \(G=\GL_n\), \(r=n\), \(t=1\), \(m=m_F\) and \(W=S_n\) (see \cite[\S4.1, \S4.2]{Serre-Gk} for the value of \(t\) and \(m\), according to the parity of \(p\)).
\end{proof}

\begin{remark}
	When \(p=2\), the assumption \(\zeta_4 \in F\) in (\ref{prop:Minkowski_bound}) is necessary:
	indeed taking \(k=F=\Rr\), we have \(m_{\Rr}=1\), while \(M_2(2,\Rr)=\infty\), since for each \(a\in \Nn\), the rotation by the angle \(2\pi/2^a\) generates a subgroup of \(\GL_2(\Rr)\) of order \(2^a\).
\end{remark}

\begin{proposition}
	\label{prop:Minkowski}
	Let \(X\) be a smooth projective \(k\)-variety with an action of a finite constant \(p\)-group \(G\).
	Set \(n = \dim X\), and suppose that \(n>0\).
	Assume that \(X\) is \(G\)-connected and has a Chern number prime to \(p\).
	Then the \(G\)-action on \(X\) factors through a quotient \(Q\) satisfying \(v_p(|Q|) \le M_p(n,k(\zeta_p))\).
	Thus, if \(p\) is odd, or \(p=2\) with \(\zeta_4 \in k\),
	\[
		v_p(|Q|) \le m_{k(\zeta_p)} \cdot n + v_p(n!).
	\]
\end{proposition}
\begin{proof}
	This follows from (\ref{cor:quotient}) and (\ref{prop:Minkowski_bound}).
\end{proof}

\begin{remark}
	\label{ex:Minkowski}
	The bound \(M_p(n,F)\) is finite when \(F\) is finitely generated over its prime field \cite[\S4.3]{Serre-Gk} (explicit formulas, including for \(p=2\), are given in \cite[\S6]{Serre-Gk}).
\end{remark}

\section{Obstructing classes of actions}
\label{sect:obstruction}

In this section \(G\) is a finite étale \(k\)-group such that \(G(\overline{k})\) is a \(p\)-group.
We fix a class \(\Pc\) of finite groups.
\begin{definition}
	\label{def:dP}
	By a quotient of \(G\) we will mean a quotient by a closed normal subgroup.
	We define the following invariant of \(G\), in \(\Nn \cup \{\infty\}\)
	\[
		d_p^{\Pc}(G)= \min \{\ed_p(Q), \text{ for \(Q\) a quotient of \(G\) with \(Q(\overline{k}) \not \in \Pc\)}\}.
	\]
	Thus \(d_p^{\Pc}(G)=\infty\) when \(G\) has no quotient \(Q\) with \(Q(\overline{k}) \not \in \Pc\).
\end{definition}

\begin{para}
	\label{p:d_(p)}
	We set
	\[
		d_{(p)}^{\Pc}(G) = \min \{d_p^{\Pc}(G_K), \text{ where \(K/k\) is finite separable of degree prime to \(p\)}\}.
	\]
\end{para}

\begin{para}
	\label{p:d_descend}
	Assume that the \(k\)-group \(G\) is finite constant.
	If \(L/k\) is a field extension of finite degree prime to \(p\), then \(d_p^{\Pc}(G)=d_p^{\Pc}(G_L)\).
	Consequently \(d_p^{\Pc}(G)=d_{(p)}^{\Pc}(G)\).

	Indeed let us fix an embedding \(L \subset \overline{k}\) over \(k\).
	Then any quotient of \(G_L\) is of the form \(Q_L\), for a quotient \(Q\) of \(G\), and \(Q_L(\overline{k})=Q(\overline{k})\).
	In addition, we have \(\ed_p(Q)= \ed_p(Q_L)\) by \cite[Proposition~1.5~(2)]{Merkurjev-essential_contemp}.
\end{para}

\begin{theorem}
	\label{th:P}
	Let \(X\) be a smooth projective \(k\)-variety with a \(G\)-action.
	Assume that \(X\) has a Chern number prime to \(p\).
	Let \(Q=G/\ker (G \to \Aut_k(X))\).
	\begin{enumerate}[label=(\roman*), ref=\roman*]
		\item
			\label{th:P:connected}
			Assume that \(X\) is geometrically connected.
			If \(\dim X <d_p^{\Pc}(G)\), then \(Q(\overline{k}) \in \Pc\).

		\item
			\label{th:P:constant}
			Assume that \(X\) is \(G\)-connected and that the \(k\)-group \(G\) is finite constant.
			If \(\dim X <d_p^{\Pc}(G)\), then \(Q(k) \in \Pc\).

		\item
			\label{th:P:etale}
			Assume that \(X\) is \(G\)-connected.
			If \(\dim X <d_{(p)}^{\Pc}(G)\), then the group \(Q(\overline{k})\) is a subdirect product of elements of \(\Pc\).
	\end{enumerate}
\end{theorem}
\begin{proof}
	\eqref{th:P:connected}:
	We apply (\ref{th:quotient}), which is possible by (\ref{p:finite_genfree}).
	As \(\ed_p(Q) \le \dim X < d_p^{\Pc}(G)\), we have \(Q(\overline{k}) \in \Pc\).

	\eqref{th:P:constant}:
	Let us apply (\ref{prop:fpt}) and use its notation.
	Since \([K:k]\) is prime to \(p\), it follows from (\ref{p:d_descend}) that \(d_{p}^{\Pc}(G)= d_p^{\Pc}(G_K)\).
	As \(\ed_p(Q') \le \dim X < d_p^{\Pc}(G)\), we have \(Q'(\overline{K}) \in \Pc\).
	Since \(Q'=Q_K\), we deduce that \(Q'(\overline{k})=Q(\overline{k})=Q(k) \in \Pc\).

	\eqref{th:P:etale}: 
	Let us apply (\ref{prop:fpt}) and use its notation.
	Since \([K:k]\) is prime to \(p\), we have \(\ed_p(Q') \le \dim X < d_{(p)}^{\Pc}(G) \le d_p^{\Pc}(G_K)\), and therefore \(Q'_{\sigma}(\overline{k}) \in \Pc\), for every morphism of \(k\)-algebras \(\sigma \colon K \to \overline{k}\), where \(Q'_{\sigma}=Q'_{\overline{k}}\) with \(K\) viewed as a subfield of \(\overline{k}\) using \(\sigma\).
	Now \(Q(\overline{k})\) is a subgroup of
	\[
		(R_{K/k}(Q'))(\overline{k}) \simeq \prod_{\sigma \in \Hom_{k-\mathrm{alg}}(K, \overline{k})} Q'_{\sigma}(\overline{k}).
	\]
	For each \(\sigma\), the map \(Q(\overline{k}) \to Q'_{\sigma}(\overline{k})\) is surjective, because \(G(\overline{k})=(G_K)_{\sigma}(\overline{k}) \to Q'_{\sigma}(\overline{k})\) is surjective and factors through \(Q(\overline{k})\).
\end{proof}

Let us now give a representation-theoretic formula for the invariant \(d_p^{\Pc}(G)\) when \(G\) is a finite constant \(p\)-group, using the computation of Karpenko and Merkurjev \cite{KM-Ess} of the essential dimension of finite \(p\)-groups.
\begin{proposition}
	\label{prop:d_p}
	Assume that the characteristic of \(k\) differs from \(p\).
	If the \(k\)-group \(G\) is a finite constant \(p\)-group, then
	\[
		d_p^{\Pc}(G) =  \min \dim_{k(\zeta_p)} V,
	\]
	where \(V\) runs over the \(G\)-representations over \(k(\zeta_p)\) such that the image \(Q\) of \(G_{k(\zeta_p)} \to \GL(V)\) satisfies \(Q(\overline{k}) \not \in \Pc\).
\end{proposition}
\begin{proof}
	Since \([k(\zeta_p):k]\) is prime to \(p\), it follows from (\ref{p:d_descend}) that \(d_p^{\Pc}(G)=d_p^{\Pc}(G_{k(\zeta_p)})\), and so we may assume that \(\zeta_p \in k\).

	Let \(V\) be a \(G\)-representation over \(k\) such that the image \(Q\) of \(G \to \GL(V)\) satisfies \(Q(\overline{k}) \not \in \Pc\).
	Then \(V\) is a faithful \(Q\)-representation.
	Therefore the \(Q\)-action on \(V\) is versal by (\ref{p:rep_versal}), and generically free by (\ref{p:finite_genfree}).
	We thus have \(\ed_p(Q) \le \dim_k V\) (see (\ref{prop:ed_versal})).
	Since \(Q(\overline{k}) \not \in \Pc\), we have by definition \(d_p^{\Pc}(G) \le \ed_p(Q)\).
	We thus obtain \(d_p^{\Pc}(G) \le \dim_k V\).

	Conversely, if \(G\) has no quotient \(Q\) such that \(Q(\overline{k}) \not \in \Pc\), then the equality of the statement holds, as both sides are infinite.
	Otherwise, let \(Q\) be a quotient of \(G\) such that \(\ed_p(Q)=d_p^{\Pc}(G)\) and \(Q(\overline{k}) \not \in \Pc\).
	By \cite[Theorem~4.1]{KM-Ess}, there exists a faithful \(Q\)-representation \(V\) over \(k\) such that \(\ed_p(Q)=\dim_k V\).
	We have thus \(d_p^{\Pc}(G) = \dim_kV\).
\end{proof}

\begin{proposition}
	\label{prop:d_p:p}
	In the situation of (\ref{prop:d_p}), assume that the class \(\Pc\) is stable under subdirect products, and contains the trivial group.
	Then
	\begin{enumerate}[label=(\roman*), ref=\roman*]
		\item
			\label{prop:d_p:p:min}
			one may additionally require \(V\) to be irreducible while computing the minimum of (\ref{prop:d_p}),
		\item
			\label{prop:d_p:p:power}
			the invariant \(d_p^{\Pc}(G)\) is either \(\infty\) or a power of \(p\).
	\end{enumerate}
\end{proposition}
\begin{proof}
	Since the trivial group belongs to \(\Pc\), each \(V\) appearing in (\ref{prop:d_p}) is nonzero, hence a nonempty direct sum of irreducible \(G\)-representations.
	Assume that \(V_1,V_2\) are \(G\)-representations over \(k\).
	Let \(V_3=V_1\oplus V_2\), and \(Q_i\) be the image of \(G \to \GL(V_i)\), for \(i=1,2,3\).
	Then \(Q_3(\overline{k})\) is a subdirect product of \(Q_1(\overline{k})\) and \(Q_2(\overline{k})\).
	Thus if \(Q_3(\overline{k}) \not \in \Pc\), then \(Q_i(\overline{k}) \not \in \Pc\) for some \(i\in \{1,2\}\).
	Since every \(G\)-representation over \(k\) is a direct sum of irreducible \(G\)-representations, this proves \eqref{prop:d_p:p:min}.
	Moreover, over a field containing a root of unity of order \(p\), irreducible representations of \(p\)-groups have dimension a power of \(p\) (see (\ref{p:V_irred}) below), and the statement \eqref{prop:d_p:p:power} follows.
\end{proof}

\section{Fixed-point theorem}
\label{sect:fpt}

In this section \(G\) is a finite étale \(k\)-group such that \(G(\overline{k})\) is a \(p\)-group, and we assume that the characteristic of \(k\) differs from \(p\).

\begin{remark}
	When \(F\) is a field of characteristic \(p\), any finite constant \(p\)-group acting on a smooth projective \(F\)-variety \(X\) with a Chern number prime to \(p\) has a nonempty fixed locus by \cite[(1.1.2)]{fpt}, regardless of the dimension of \(X\).
\end{remark}

\begin{para}
	\label{def:d_p}
	Let us use the notation of (\ref{def:dP}) for the class \(\Pc=\ab\) of finite abelian groups.
	This class is stable under finite direct products and subgroups, hence it is stable under subdirect products.
	By definition
	\[
		\dab(G) = \min \{\ed_p(Q), \text{ for \(Q\) a quotient of \(G\) with \(Q(\overline{k})\) not abelian}\},
	\]
	and \(\dab(G)=\infty\) when \(G(\overline{k})\) is abelian.
\end{para}

\begin{remark}
	Let \(X\) be a smooth, projective, geometrically connected \(k\)-variety with a \(G\)-action.
	Assume that \(X\) has a Chern number prime to \(p\), and that \(\dim X < \dab(G)\).
	Then (\ref{th:P}.\ref{th:P:connected}) implies that the \(G\)-action on \(X\) factors through a quotient \(Q\) such that \(Q(\overline{k})\) is an abelian \(p\)-group.
\end{remark}

In the next theorem, we use the notation introduced in (\ref{p:d_(p)}).
\begin{theorem}
	\label{th:fpt}
	Let \(X\) be a smooth projective \(k\)-variety with a \(G\)-action.
	Assume that \(X\) has a Chern number prime to \(p\), and that \(\dim X < \dabb(G)\).
	Then \(X^G \ne \varnothing\).
\end{theorem}
\begin{proof}
	Since a Chern number of \(X\) is prime to \(p\), there exists a \(G\)-connected closed and open subscheme \(X'\) of \(X\) having a Chern number prime to \(p\).
	As \(X'^G \subset X^G\), we may replace \(X\) with \(X'\), and thus assume that \(X\) is \(G\)-connected.

	Let us apply (\ref{th:P}.\ref{th:P:etale}) and use its notation.
	We obtain that \(Q=G/\ker (G\to \Aut_k(X))\) is such that \(Q(\overline{k})\) is a subdirect product of finite abelian groups, hence is abelian.
	Being a quotient of \(G(\overline{k})\), the group \(Q(\overline{k})\) is thus a finite abelian \(p\)-group.
	To show that \(X^G\ne \varnothing\), we may extend the base field (which changes neither the Chern numbers of \(X\) nor the emptiness of the scheme \(X^G\)), and thus assume that the \(k\)-group \(Q\) is diagonalizable.
	Then \cite[(4.4)]{fpt} implies that \(X^G=X^Q\ne \varnothing\).
\end{proof}

\begin{para}
	\label{rem:d_p_power}
	Assume that the \(k\)-group \(G\) is a finite constant \(p\)-group.
	Since subgroups of \(\GL(V)\) are commutative when \(\dim_kV=1\), it follows from (\ref{prop:d_p:p}.\ref{prop:d_p:p:power}) that the invariant \(\dab(G)\) is either \(\infty\) (if \(G\) is abelian), or a nontrivial power of \(p\) (otherwise).
	In particular we have \(\dab(G) \ge p\).
\end{para}

\begin{remark}
	Since \(\dab(G)=d_{(p)}^{\ab}(G) \ge p\) (see (\ref{p:d_descend}) and (\ref{rem:d_p_power})), Theorem~(\ref{th:fpt}) recovers \cite[(1.1.2.iii)]{fpt} as a special case.
\end{remark}

\begin{example}
	Assume that \(k\) is algebraically closed (recall that its characteristic is different from \(p\)).
	If \(G\) is an extraspecial \(p\)-group of order \(p^{1+2r}\), then \(\dab(G)=p^r\) by (\ref{prop:d_p:p}.\ref{prop:d_p:p:min}), since irreducible representations of \(G\) have dimension \(1\) or \(p^r\) (taking (\ref{rem:d_p_power}) into account).
\end{example}

\begin{example}
	\label{ex:d_2_ext}
	The invariant \(\dab(G)\) can change under extension of the base field.
	For example, let \(G\) be the extraspecial \(2\)-group of minus type of order \(32\) over \(\Qq\).
	Then
	\[
		d_2^{\ab}(G)=8 \quad \text{ and } \quad d_2^{\ab}(G_{\Qq(i)})=4.
	\]
	Indeed every nontrivial normal subgroup of \(G\) contains the center of \(G\), so proper quotients of \(G\) are abelian.
	Therefore any \(G\)-representation not factoring through an abelian quotient of \(G\) is faithful.
	Over \(\overline{\Qq}\) the group \(G\) has a unique faithful irreducible representation \(W\), of dimension \(4\).
	Since \(G\) is of minus type, the representation \(W\) is quaternionic, hence its Schur index over \(\Rr\), and a fortiori over \(\Qq\), is at least \(2\).
	On the other hand \(G\) is isomorphic to the central product of \(Q_8\) and \(D_8\) (the quotient of \(Q_8\times D_8\) identifying their centers), and \(W\) is the tensor product of the faithful \(2\)-dimensional representations of \(Q_8\) and \(D_8\), which are defined over \(\Qq(i)\) and \(\Qq\) respectively.
	Thus \(W\) is defined over \(\Qq(i)\), and its Schur index over \(\Qq\) is \(2\).
	In the notation of (\ref{p:V_irred}), the irreducible faithful \(\Qq[G]\)-module \(V\) has \(s=1\), \(w=4\) and \(m=2\), so \(\dim_{\Qq}V=8\) and \(d_2^{\ab}(G)=8\).
	Since \(W\) is defined over \(\Qq(i)\), we have \(d_2^{\ab}(G_{\Qq(i)}) = 4\). 
\end{example}

\begin{remark}
	The fixed-point theorem (\ref{th:fpt}) is sensitive to the base field.

	Indeed, consider for instance the extraspecial \(2\)-group \(G\) of order \(32\) and minus type over \(\Qq\) of (\ref{ex:d_2_ext}).
	Then \(d_2^{\ab}(G)=8\), and so by (\ref{th:fpt})  and (\ref{p:d_descend}) no smooth projective \(\Qq\)-variety \(X\) with \(\dim X< d_2^{\ab}(G)\) having an odd Chern number carries a \(G\)-action such that \(X^G=\varnothing\).

	On the other hand \(d_2^{\ab}(G_{\Qq(i)})=4\), and \(G_{\Qq(i)}\) has an irreducible representation \(V\) of dimension \(4\).
	Then \(G_{\Qq(i)}\) acts on the \(\Qq(i)\)-variety \(Y=\Pp(V\oplus 1 \oplus 1 \oplus 1)\) with fixed locus \(\Pp(1 \oplus 1 \oplus 1)=\Pp^2\).
	Let \(X\) be the blow-up of \(Y\) at its fixed locus.
	Then \(G_{\Qq(i)}\) acts on the \(\Qq(i)\)-variety \(X\) with fixed locus \(\Pp_{\Pp^2}(V(1))^{G_{\Qq(i)}}=\varnothing\) (as seen by applying (\ref{lemm:fixed_proj}) over an affine Zariski cover of \(\Pp^2\)).
	In addition \(X\) has an odd Chern number, \(\deg c_3(\Tan_X)^2\), and \(\dim X=6 <8\).
	(A more sophisticated construction will yield in (\ref{prop:sharp_bound}) an example of dimension \(4\).)
\end{remark}

\section{Birational automorphisms}
\label{sect:birat}
\numberwithin{theorem}{section}
\numberwithin{lemma}{section}
\numberwithin{proposition}{section}
\numberwithin{corollary}{section}
\numberwithin{example}{section}
\numberwithin{definition}{section}
\numberwithin{remark}{section}

We now discuss some consequences for groups of birational automorphisms of rationally connected varieties.
In this section \(k\) denotes a field of characteristic zero.

\begin{para}
	\label{p:regularization}
	Let \(X\) be a smooth, projective, geometrically connected \(k\)-variety.
	By \cite[Lemma-Definition~3.1]{Prokhorov-Shramov}, given a finite group \(G\) of birational automorphisms of \(X\), we may find a smooth, projective, geometrically connected \(k\)-variety \(Y\) with a \(G\)-action, such that \(X\) is \(G\)-equivariantly birational to \(Y\).
	The group \(G\) then acts faithfully on \(Y\), since the triviality of an automorphism of \(Y\) may be checked on a dense open subscheme.
	Then \(\chi(Y,\Oc_Y)=\chi(X,\Oc_X)\) (see e.g.\ \cite[Theorem~3.2.8]{Higherdirect}).
\end{para}

The \emph{rank} of a finite abelian group is its minimal number of generators.
Here is a generalization of the result of \cite{Xu-rank_bir}:
\begin{proposition}
	\label{prop:birat}
	Let \(X\) be a smooth, projective, geometrically connected \(k\)-variety such that \(\chi(X,\Oc_X)\) is prime to \(p\).
	Assume that \(\dim X < p-1\).
	Then any finite \(p\)-group of birational automorphisms of \(X\) is abelian of rank \(\le \dim X\), and of exponent \(< p^s\) if \(\zeta_{p^s} \not \in k(\zeta_p)\).
\end{proposition}
\begin{proof}
	Let \(G\) be such a group, viewed as a finite constant \(k\)-group, and let \(Y\) be as in (\ref{p:regularization}), so that \(G\) acts faithfully on \(Y\) and \(\chi(Y,\Oc_Y)\) is prime to \(p\).
	As \(\dim Y <p-1\), this implies that some Chern number of \(Y\) is prime to \(p\) (see e.g.\ \cite[Theorem~5.1~(iii)]{invariants}).

	Since \(\dim Y <p-1 < p\le \dab(G)\) by (\ref{rem:d_p_power}), we can apply (\ref{th:P}.\ref{th:P:connected}) with \(\Pc=\ab\) to the \(k\)-variety \(Y\) to deduce that \(G\) is an abelian \(p\)-group, and so \(G \simeq \Zz/p^{n_1}\times \dotsb \times \Zz/p^{n_r}\), for some \(n_1,\dotsc,n_r \ge 1\).
	Then (\ref{prop:d_p_constant}) yields 
	\[
		[k(\zeta_{p^{n_1}}):k(\zeta_p)] + \dotsb + [k(\zeta_{p^{n_r}}):k(\zeta_p)] \le \dim Y=\dim X.
	\]
	Since \(\dim X < p\) and since the numbers \([k(\zeta_{p^{n_i}}):k(\zeta_p)]\) are powers of \(p\), they must all equal \(1\), and their sum \(r\) satisfies \(r \le \dim X\).
	This means that \(\zeta_{p^{n_i}} \in k(\zeta_p)\) for all \(i\).
	If \(\zeta_{p^s}\not \in k(\zeta_p)\), then \(n_i< s\) for each \(i=1,\dotsc, r\), which yields the statement.
\end{proof}

\begin{remark}
	When \(X\) is a smooth projective connected \(k\)-variety, the integer \(\chi(X,\Oc_X)\) needs to be multiplied by the denominator of the Todd class, in order to obtain a Chern number.
	When \(\dim X\ge p-1\), this denominator is divisible by \(p\) and therefore \(\chi(X,\Oc_X)\) can no longer be used to produce a Chern number prime to \(p\).
	It does however produce a Chern number whose \(p\)-adic valuation can be controlled, a fact which we will exploit in (\ref{p:Todd}).

	In this situation, and for actions of abelian \(p\)-groups, one can still obtain a rank bound using the number \(\chi(X,\Oc_X)\):
	indeed Kollár--Zhuang \cite[Corollary~9]{Kollar-Zhuang} showed that, if a finite abelian \(p\)-group \(G\) of rank \(r\) acts faithfully on a smooth projective connected \(k\)-variety \(X\) of dimension \(n\) over a field of characteristic different from \(p\), then
	\[
		r \le v_p(\chi(X,\Oc_X)) + \frac{p}{p-1} n.
	\]
	Although connectedness is not explicitly assumed in \cite[Corollary~9]{Kollar-Zhuang}, it holds in all applications given there.
	It is necessary:
	for \(p\) odd with \(\zeta_p \in k\), the group \((\Zz/p)^2\) acts faithfully on \(X=\Pp^1 \sqcup \Pp^1\), the \(i\)-th \(\Zz/p\) factor acting on the \(i\)-th copy of \(\Pp^1\) via \([x:y] \mapsto [x:\zeta_py]\) and trivially on the other.
	Here \(n=1\) and \(\chi(X,\Oc_X)=2\), so that
	\[
		r=2 > \frac{3}{2}\ge  0 + \frac{p}{p-1}=v_p(\chi(X,\Oc_X)) +  \frac{p}{p-1} n.
	\]
\end{remark}

\begin{remark}
	Let \(X\) be a smooth, projective, geometrically connected \(k\)-variety.
	If \(X\) is rationally connected, then \(\chi(X,\Oc_X)=1\) (the groups \(H^i(X,\Oc_X)\) vanish for \(i>0\) by \cite[Corollary~4.18, a)]{Debarre-book}, and \(H^0(X,\Oc_X)=k\) as the \(k\)-variety \(X\) is projective and geometrically connected), and so (\ref{prop:birat}) applies to \(X\).
	This recovers the result proved in \cite{Xu-rank_bir} when \(k\) is algebraically closed.
	Let us mention that Xu deduced the commutativity of the action from the existence of a fixed point, based on the fixed-point theorem of \cite{fpt}; here the logic is reversed.\\
\end{remark}

Proposition (\ref{prop:birat}) applies in particular to \(X=\Pp^n\), for \(n < p-1\), whose group of birational automorphisms is the Cremona group \(\Cr_n(k)\).
In particular, any finite \(p\)-subgroup \(G\) of \(\Cr_n(k)\), with \(n < p-1\), is abelian of rank \(\le n\) (this was proved in \cite{Xu-rank_bir}).
But we also get:

\begin{proposition}
	\label{prop:Cremona}
	Let \(s\in \Nn\) such that \(\zeta_{p^s} \not \in k(\zeta_p)\).
	If \(n< p-1\), then \(\Cr_n(k)\) contains no element of order \(p^s\).
\end{proposition}
\begin{proof}
	As \(\chi(\Pp^n,\Oc_{\Pp^n})=1\), this follows from (\ref{prop:birat}).
\end{proof}

\begin{corollary}
	When \(n< p -1\), the group \(\Cr_n(\Qq)\) contains no element of order \(p^2\).
\end{corollary}

\begin{remark}
	In the case \(n=2\), Serre showed in \cite{Serre-Minkowski} that the order of every finite subgroup of \(\Cr_2(\Qq)\) divides \(2^7\cdot 3^3 \cdot 5 \cdot 7\); in particular there are no elements of order \(5^2\) and \(7^2\), and for \(p\ge 11\) there are no elements of order \(p\) at all.
\end{remark}

\section{Derived length}
\label{sect:derived}
In this section we apply the results of \S\ref{sect:obstruction} to obtain results on the derived length of group actions on varieties, generalizing the case of abelian actions (that is, those of derived length \(\le 1\)).
We still assume that \(k\) is a field of characteristic different from \(p\).

\begin{para}
	For a finite group \(G\), let us define inductively \(G^{(0)}=G\), and \(G^{(n+1)}=[G^{(n)},G^{(n)}]\).
	The \emph{derived length} of \(G\) is the least integer \(n \in \Nn\) such that \(G^{(n)}=1\).
	Note that \(G\) is abelian if and only if its derived length is \(\le 1\).
\end{para}

\begin{para}
	Let us use the notation of (\ref{def:dP}) for the class \(\Pc=\der_\ell\) of finite groups of derived length \(\le \ell\).
	This class is stable under finite direct products and subgroups, hence it is stable under subdirect products.
	By definition
	\[
		d_p^{\der_\ell}(G) = \min \{\ed_p(Q), \text{ for \(Q\) a quotient of \(G\) such that \(Q(\overline{k})^{(\ell)}\ne 1\)}\},
	\]
	and \(d_p^{\der_\ell}(G)=\infty\) when \(G(\overline{k})\) has derived length \(\le \ell\).
	We will also use the notation \(d_{(p)}^{\der_\ell}(G)\) introduced in (\ref{p:d_(p)}).
\end{para}

\begin{proposition}
	\label{prop:derived}
	Let \(G\) be a finite étale \(k\)-group such that \(G(\overline{k})\) is a \(p\)-group.
	Let \(X\) be a smooth projective \(k\)-variety with a \(G\)-action.
	Assume that \(X\) is \(G\)-connected, and has a Chern number prime to \(p\).
	If \(\dim X <d_{(p)}^{\der_\ell}(G)\), then \(G\) acts on \(X\) through a quotient of derived length \(\le \ell\).
\end{proposition}
\begin{proof}
	It follows from (\ref{th:P}.\ref{th:P:etale}) that \(G\) acts on \(X\) through a quotient \(Q\) such that \(Q(\overline{k})\) is a subdirect product of finite groups of derived length \(\le \ell\), hence \(Q(\overline{k})\) has derived length \(\le \ell\).
\end{proof}

\begin{lemma}
	\label{lemm:derived_length}
	Let \(G\) be a finite \(p\)-group, and \(V\) a \(k[G]\)-module.
	If \(\dim_k V < p^\ell\), then \(G^{(\ell)}\) acts trivially on \(V\).
\end{lemma}
\begin{proof}
	We proceed by induction on \(\ell\).
	The case \(\ell=0\) forces \(V=0\), so let us assume that \(\ell>0\).
	We may extend scalars to an algebraic closure, and assume that \(k\) is algebraically closed.
	The \(k[G]\)-module \(V\) splits as a direct sum of irreducible \(k[G]\)-modules, and we may thus assume that \(V\) is irreducible.

	If \(\dim_k V=1\), then \(G^{(1)}=[G,G]\) acts trivially on \(V\) by (\ref{lemm:inv_ab}), hence so does \(G^{(\ell)}\subset G^{(1)}\).

	Assume that \(\dim_k V>1\).
	By (\ref{lemm:Clifford}), there exists a normal subgroup \(H\) of \(G\) of index \(p\) and a \(k[H]\)-module \(W\) such that \(V \simeq \Ind^G_HW\).
	The \(k[H]\)-module \(V\) decomposes as a direct sum of the modules \(gW\) where \(g\in G/H\).
	Applying the induction hypothesis to each \(gW\), whose dimension as a \(k\)-vector space is \((\dim_kV)/p<p^{\ell-1}\), we obtain that \(H^{(\ell-1)}\) acts trivially on each \(gW\), and thus on \(V\).
	Since \(G/H\) has order \(p\), it is abelian, so that \(G^{(1)} \subset H\).
	Then \(G^{(\ell)}=(G^{(1)})^{(\ell-1)} \subset H^{(\ell-1)}\) acts trivially on \(V\).
\end{proof}

\begin{para}
	\label{rem:d_p_der_power}
	Assume that the \(k\)-group \(G\) is a finite constant \(p\)-group.
	Since \(p\)-subgroups of \(\GL(V)\) have derived length \(\le \ell\) when \(\dim_kV<p^\ell\) by (\ref{lemm:derived_length}), it follows from (\ref{prop:d_p:p}.\ref{prop:d_p:p:power}) that the invariant \(d_p^{\der_\ell}(G)\) is either \(\infty\) (if \(G^{(\ell)}=1\)), or a power of \(p\) no less than \(p^\ell\) (otherwise).
	In particular we have \(d_p^{\der_\ell}(G) \ge p^\ell\).
\end{para}

\begin{remark}
	Assume that the \(k\)-group \(G\) is a finite constant \(p\)-group.
	Then \(d_p^{\der_\ell}(G)=d_{(p)}^{\der_\ell}(G)\) by (\ref{p:d_descend}).
	By (\ref{rem:d_p_der_power}), Proposition~(\ref{prop:derived}) applies as soon as \(\dim X <p^\ell\), yielding a higher analogue of \cite[(1.1.2.iii)]{fpt}.
\end{remark}

\section{Constructing fixed-point-free actions}
\label{sect:construction}

In this section, we analyze the sharpness of the dimensional bound in (\ref{th:fpt}).
Here \(G\) denotes a finite \(p\)-group, which will be viewed as a finite constant \(k\)-group, and we assume that the characteristic of \(k\) differs from \(p\).

We call a \(G\)-action on a \(k\)-variety \(X\) \emph{fixed-point-free} if the fixed locus \(X^G\) is empty.
We will prove the following:
\begin{proposition}
	\label{prop:sharp_bound}
	The minimum of the dimensions of the smooth projective \(k\)-varieties with a fixed-point-free \(G\)-action and a Chern number prime to \(p\) is
	\begin{itemize}
		\item[---] \(4\) if \(p=2\) and \(d_2^{\ab}(G)=2\),
		\item[---] \(d_p^{\ab}(G)\) otherwise.
	\end{itemize}
	(No such variety exists if and only if \(d_p^{\ab}(G)=\infty\), i.e.\ \(G\) is abelian.)
\end{proposition}

We will use the following device to eliminate fixed points.
\begin{lemma}
	\label{lemm:fixed_blowup}
	Let \(X\) be a smooth \(k\)-variety with an action of \(G\).
	If \(X^G\) is a single rational point \(x\) and \((\Tan_{X,x})^{[G,G]}=0\), then the blow-up of \(X\) at \(x\) is fixed-point-free.
\end{lemma}
\begin{proof}
	The fixed locus of the blow-up is contained in the \(k\)-scheme \(\Pp(\Tan_{X,x})\), since the complement of the latter is \(G\)-equivariantly isomorphic to \(X\smallsetminus x=X\smallsetminus X^G\).
	By (\ref{lemm:fixed_proj}) we have \(\Pp(\Tan_{X,x})^G=\Pp(0)^G=\varnothing\).
\end{proof}

\noindent \emph{Setup for the proof of (\ref{prop:sharp_bound}).}
It is proved in \cite[(1.1.4)]{fpt} that every smooth projective \(k\)-variety of dimension \(<4\) with an action of a constant \(2\)-group and an odd Chern number has a nonempty fixed locus.
Combining this observation with (\ref{th:fpt}) (recall that \(d_p^{\ab}(G)=d_{(p)}^{\ab}(G)\) by (\ref{p:d_descend})), we see that it will suffice to assume that \(d_p^{\ab}(G) \ne \infty\), and construct a smooth projective \(k\)-variety of dimension \(d_p^{\ab}(G)\), resp.\ \(4\) when \(p=2\) and \(d_2^{\ab}(G)=2\), admitting a fixed-point-free \(G\)-action and having a Chern number prime to \(p\).

Let us first observe that we may assume that \(\zeta_p \in k\).
Indeed for \(p=2\), there is nothing to prove.
If \(p\) is odd, assume that \(X\) is a smooth projective \(k(\zeta_p)\)-variety of dimension \(d_p^{\ab}(G_{k(\zeta_p)})\) with a fixed-point-free \(G_{k(\zeta_p)}\)-action, and having a Chern number prime to \(p\).
Then \(X\) viewed as a \(k\)-variety carries a fixed-point-free \(G\)-action, where \(G\) acts trivially on \(\Spec k(\zeta_p)\);
it is smooth projective of dimension \(d_p^{\ab}(G_{k(\zeta_p)})\), and its Chern numbers are those of the \(k(\zeta_p)\)-variety \(X\) multiplied by \([k(\zeta_p):k]\).
Since \([k(\zeta_p):k]\) is prime to \(p\), this \(k\)-variety has a Chern number prime to \(p\), and its dimension equals \(d_p^{\ab}(G)\), by (\ref{p:d_descend}).
We may thus replace \(k\) with \(k(\zeta_p)\), and assume that \(\zeta_p \in k\).

By (\ref{prop:d_p}) and (\ref{prop:d_p:p}.\ref{prop:d_p:p:min}), there exists an irreducible \(k[G]\)-module \(V\) of dimension \(d_p^{\ab}(G)\) on which \(G\) does not act through an abelian quotient.
In other words \([G,G]\) acts nontrivially on \(V\).
The irreducible summands of the \(\overline{k}[G]\)-module \(V_{\overline{k}}\) all have the same dimension \(w\), see (\ref{p:V_irred}).
If \(w=1\), then \([G,G]\) acts trivially on \(V_{\overline{k}}\), hence on \(V\), contradicting the assumption.
Therefore no summand of the \(\overline{k}[G]\)-module \(V_{\overline{k}}\) has dimension \(1\).
Thus by (\ref{p:ab_dim_one}) we have \(V^{[G,G]}=0\).

\begin{proof}[Proof  of (\ref{prop:sharp_bound}) when \(p=2\) and \(d_2^{\ab}(G)=2\)]
	The group \(G\) acts on \(\Pp(V \oplus 1) \times_k\Pp(V \oplus 1)\) with fixed locus the single rational point \(\Pp(1) \times_k \Pp(1)\), by (\ref{lemm:fixed_proj}), since \(V^{[G,G]}=0\).
	The tangent space at this point is \(G\)-equivariantly isomorphic to \(V \oplus V\).
	Blowing up this point yields a smooth projective \(k\)-variety of dimension \(4\) with a fixed-point-free \(G\)-action, by (\ref{lemm:fixed_blowup}).
	In addition the blow-up has an odd Chern number, as computed in \cite[(4.5)]{fpt}.
\end{proof}

\begin{proof}[Proof  of (\ref{prop:sharp_bound}) when \(p \ne 2\)]
	The group \(G\) acts on \(\Pp(V\oplus 1)\) with fixed locus the single rational point \(\Pp(1)\), by (\ref{lemm:fixed_proj}), since \(V^{[G,G]}=0\).
	The tangent space at this point is \(G\)-equivariantly isomorphic to \(V\).
	Blowing up this point yields a smooth projective \(k\)-variety of dimension \(d_p^{\ab}(G)\) with a fixed-point-free \(G\)-action, by (\ref{lemm:fixed_blowup}).
	The blow-up has a Chern number prime to \(p\), as computed in \cite[(4.5)]{fpt}.
\end{proof}

\begin{proof}[Proof  of (\ref{prop:sharp_bound}) when \(p=2\) and \(d_2^{\ab}(G) > 2\)]
	We have \(d_2^{\ab}(G) \ge 4\) by (\ref{rem:d_p_power}).
	We apply (\ref{lemm:semilinear}) below and use its notation.
	We may view \(V\) as a \(G\)-equivariant vector bundle of rank \(d/2\) over \(B=\Spec L\).
	We consider the \(L\)- and \(k\)-schemes
	\[
		P= \Pp_B(V \oplus 1), \quad \text{ and } Y=R_{L/k} P.
	\]
	(Here \(R_{L/k}\) refers to Weil restriction of scalars; it is proved in \cite[\S7.6, Theorem~4]{Neron} that the functor \(Y\) is representable by a \(k\)-scheme.)
	The \(\overline{k}\)-scheme \(Y_{\overline{k}}\) is isomorphic to \(P_1 \times_{\overline{k}} P_2\), where the \(\overline{k}\)-schemes \(P_1,P_2\) are the base changes of \(P\) under the two geometric points \(L \to \overline{k}\).
	Therefore the \(k\)-scheme \(Y\) is smooth \stacks{02UQ} and projective \stacks{0BDG}.

	The \(k\)-group \(G\) naturally acts on \(Y\);
	explicitly the action of an element \(g \in G(k)\) on an element \(y\in Y(S)=P(S\times_k B)\), for \(S\) a \(k\)-scheme, is the element of \(P(S\times_k B)\) given by the morphism of \(B\)-schemes
	\[
		S \times_k B \xrightarrow{ \id_S \times g^{-1}}S \times_k B \xrightarrow{y} P \xrightarrow{g} P.
	\]

	We apply (\ref{lemm:fixed_proj}) to the \(L[H]\)-module \(V \oplus 1\), and obtain that \(P^H=\Pp_B(1)\) (recall that \(V^{[H,H]}=0\); viewing \(V\) as a \(k[H]\)- or \(L[H]\)-module does not affect this fact).
	For a \(k\)-scheme \(S\), an \(S\)-point of \(Y^H\) is an \((S\times_k B)\)-point of \(P\) fixed by \(H\), the group \(H\) acting trivially on \(B\).
	It follows that \(Y^H=R_{L/k}(P^H)\).
	As \(R_{L/k}(B)=\Spec k\), we deduce that \(Y^H\) is the rational point \(y\) induced by \(B=\Pp_B(1) \to P\).
	Now \(Y^G \subset Y^H\).
	Since \(H\) is normal in \(G\), the closed subscheme \(Y^H\) is \(G\)-invariant;
	but \(Y^H=y\) and \(\Aut_k(y)=\Aut_k(\Spec k)=1\), hence \(G\) acts trivially on \(Y^H\), and therefore \(y=Y^G\).\\

	Via the dual numbers interpretation, we see that, for a \(k\)-scheme \(S\), an \(S\)-point of \(\Tan_{Y/k}\) is an \(S_L\)-point of \(\Tan_{P/L}\), giving a \(G\)-equivariant isomorphism of \(Y\)-schemes \(\Tan_{Y/k}\simeq R_{L/k}(\Tan_{P/L})\).
	Since \(\Tan_{P/L}|_{\Pp_B(1)}=V\), it follows that the tangent space of \(Y\) at \(y\) is \(G\)-equivariantly isomorphic to \(R_{L/k}(V)\) as a \(k\)-scheme (here we have used that the Weil restriction maps cartesian squares of \(L\)-schemes into cartesian squares of \(k\)-schemes, which is immediate from the functorial description).
	As the Weil restriction of a vector bundle over \(L\) is the underlying \(k\)-vector space, we have an isomorphism of \(k[G]\)-modules \(\Tan_{Y,y} \simeq V\).

	Let \(X\) be the blow-up of \(Y\) at the point \(y\), a smooth projective \(k\)-variety with a \(G\)-action.
	It is fixed-point-free by (\ref{lemm:fixed_blowup}), since \((\Tan_{Y,y})^{[G,G]} \simeq V^{[G,G]} =0\).

	Finally \(X_{\overline{k}}\) is the blow-up of \(\Pp^{d/2} \times \Pp^{d/2}\) at a rational point, and so, as explained in \cite[(4.5)]{fpt}, it has an odd Chern number.
	Its dimension is \(d\).
\end{proof}

\begin{lemma}
	\label{lemm:semilinear}
	Assume that \(p=2\).
	Let \(V\) be an irreducible \(k[G]\)-module such that \(V^{[G,G]}=0\).
	Set \(d=\dim_k V\), and assume that \(d\ge 4\).
	Then there exist
	\begin{itemize}
		\item a quadratic étale \(k\)-algebra \(L\),
		\item a group morphism \(\tau\colon G \to \Aut_k(L)\),
		\item an \(L\)-module structure on \(V\),
	\end{itemize}
	such that
	\begin{itemize}
		\item the \(L\)-module \(V\) is free of rank \(d/2\),
		\item the \(G\)-action on \(V\) is \(\tau\)-semilinear,
		\item letting \(H=\ker \tau\), we have \(V^{[H,H]}=0\).
	\end{itemize}
\end{lemma}
\begin{proof}
	Let us use the notation and observations of (\ref{p:V_irred}).
	By Roquette's theorem \cite[(74.15)]{Curtis-Reiner-Methods-II}, the Schur index \(m\) of \(V\) over \(k\) is \(1\) or \(2\) when \(k\) has characteristic zero.
	When \(k\) has positive characteristic, we have \(m=1\).
	Indeed \(k[G]\simeq k_0[G] \otimes_{k_0} k\) where \(k_0\) is the prime subfield of \(k\) (a finite field), and since \(k_0[G]\) is a product of matrix algebras over finite fields, the \(k\)-algebra \(k[G]\) is a product of matrix algebras over étale \(k\)-algebras, and thus a product of matrix algebras over fields, which implies that \(D=\End_{k[G]}(V)\) is commutative.\\

	\emph{Case \(D\ne k\):}
	We first claim that the \(k\)-algebra \(D\) contains a quadratic field extension \(L/k\).
	This is true when \(s \ne 1\), i.e.\ \(K\ne k\), because \(K/k\) is Galois with automorphism group an abelian \(2\)-group and thus admits a separable quadratic subextension.
	Otherwise \(s=1\), i.e.\ \(K=k\);
	as \(D\ne k\), we have \(m=2\), and \(D\) is a quaternion \(k\)-algebra.
	Since \(k\) has characteristic different from \(2\), the \(k\)-algebra \(D\) contains a separable quadratic extension of \(k\).

	We then define \(\tau\) as the trivial morphism.
	Since \(L\subset D=\End_{k[G]}(V)\), the \(D\)-module \(V\) is an \(L\)-vector space;
	its dimension equals \(d /2\).
	In addition the \(G\)-action on \(V\) is \(L\)-linear, and thus \(\tau\)-semilinear for our trivial \(\tau\).
	Since \(H=G\), the hypothesis implies that \(V^{[H,H]}=0\).\\

	\emph{Case \(D = k\):} Then \(V\) is absolutely irreducible.
	By (\ref{lemm:Clifford}) there exists a normal subgroup \(H\) of index \(2\) in \(G\), and an irreducible \(\overline{k}[H]\)-module \(W\) such that \(V_{\overline{k}} =\Ind^G_HW\).
	Let \(L=\End_{k[H]}(V)\), and pick \(g\in G\smallsetminus H\).
	Then \(V_{\overline{k}} \simeq W \oplus gW\) as \(\overline{k}[H]\)-modules.
	In addition, by Frobenius reciprocity
	\[
		\End_{\overline{k}[G]}(V_{\overline{k}}) \simeq \Hom_{\overline{k}[H]}(W, V_{\overline{k}})
		\simeq \End_{\overline{k}[H]}(W) \oplus \Hom_{\overline{k}[H]}(W,gW).
	\]
	Since the left-hand side is \(\overline{k}\), and \(\overline{k}\subset \End_{\overline{k}[H]}(W)\), we must have \(\Hom_{\overline{k}[H]}(W,gW)=0\).
	It follows that
	\[
		L\otimes_k \overline{k}=\End_{\overline{k}[H]}(V_{\overline{k}}) \simeq \End_{\overline{k}[H]}(W) \times \End_{\overline{k}[H]}(gW) \simeq \overline{k} \times \overline{k}.
	\]
	Therefore \(L\) is a quadratic étale \(k\)-algebra, naturally acting on \(V\).

	Since \(H\) is normal in \(G\), the group \(G\) acts by conjugation on \(\End_{k[H]}(V)=L\), thus defining \(\tau\), and making the \(G\)-action on \(V\) semilinear.
	If \(\tau\) were trivial, then the action of \(G\) on \(V\) would commute with the elements of \(L\), that is, \(L \subset \End_{k[G]}(V)=k\), a contradiction.
	So \(\tau\) is surjective.
	Certainly \(H \subset \ker \tau\); since both subgroups have index \(2\) in \(G\), they coincide.
	In particular \(\sigma=\tau(g)\) is the nontrivial automorphism of \(L\).
	If \(e_1, e_2\) are the two idempotents of \(L\otimes_k \overline{k}\), then \(e_2=\sigma(e_1)\) and so \(g\) induces an isomorphism between \(e_1V_{\overline{k}}\) and \(e_2V_{\overline{k}}\).
	Therefore the \(L\otimes_k \overline{k}\)-module \(V_{\overline{k}}\) is free of rank \(\dim_{\overline{k}}(e_1V_{\overline{k}})=d/2\), and by descent the same is true for the \(L\)-module \(V\).
	We have seen that the irreducible summands of the \(\overline{k}[H]\)-module \(V_{\overline{k}}\) are \(W\) and \(gW\).
	Their common dimension is \(d/2\ge 2\).
	By (\ref{p:ab_dim_one}) this implies that \(V^{[H,H]}=0\).
\end{proof}

\section{The configuration space}
\label{sect:conf}

In this section, we fix a \(k\)-group \(G\).

\begin{definition}
	\label{def:conf}
	Let \(X\) be a \(k\)-variety, and \(d\in \Nn\smallsetminus \{0\}\).
	For a \(k\)-scheme \(S\), we set
	\[
		\Conf{d}{X}(S) = \{\text{closed subschemes \(Z \subset X_S\) finite étale of degree \(d\) over \(S\)}\}.
	\]
	This defines a functor from \(k\)-schemes to sets.
\end{definition}

\begin{proposition}
	\label{prop:conf_rep}
	The functor \(\Conf{d}{X}\) is represented by a \(k\)-variety.
	In addition:
	\begin{enumerate}[label=(\roman*), ref=\roman*]
		\item 
			\label{prop:conf_rep:smooth}
			If \(X\) is smooth over \(k\), then so is \(\Conf{d}{X}\).
		\item 
			\label{prop:conf_rep:dim}
			If \(X\) has pure dimension \(n\), then \(\Conf{d}{X}\) has pure dimension \(dn\).
		\item
			\label{prop:conf_rep:dense}
			If \(U\) is a dense open subscheme of \(X\), then \(\Conf{d}{U}\) is a dense open subscheme of \(\Conf{d}{X}\).
	\end{enumerate}
\end{proposition}
\begin{proof}
	Let \(\Delta_X \subset X^d\) be the closed subset defined as the union of the partial diagonals \(\{(x_1,\dotsc,x_d) \mid x_i =x_j\}\) for \(i \neq j\).

	The finite constant \(k\)-group \(S_d\) acts freely on the \(k\)-variety \(\Uc_X=X^d \smallsetminus \Delta_X\), hence by (\ref{p:quotient_finite}) the fppf quotient \(\Uc_X/S_d\) is represented by a \(k\)-variety.
	We now identify the functors \(\Uc_X/S_d\) and \(\Conf{d}{X}\).

	Let \(S\) be a \(k\)-scheme, and assume that \(Z \in \Conf{d}{X}(S)\).
	Let \(S'=\Isom_S(\coprod_{i=1}^d S, Z)\), so that \(S' \to S\) is an \(S_d\)-torsor.

	The \(d\) sections \(S' \to Z \times_S S' \to X_{S'}\) have pairwise disjoint images, and thus give rise to a morphism \(S' \to \Uc_X\).
	This morphism is \(S_d\)-equivariant, and descends to a morphism of the \(S_d\)-quotients \(S \to \Uc_X/S_d\).

	Conversely, a morphism \(S \to \Uc_X/S_d\) defines by base change an \(S_d\)-equivariant morphism \(S'=S \times_{\Uc_X/S_d} \Uc_X \to \Uc_X\).
	This gives rise to \(d\) pairwise disjoint sections \(S' \to X_{S'}\), and thus to a closed immersion \(\coprod_{i=1}^ dS' \hookrightarrow X_{S'}\).
	The corresponding closed subscheme is \(S_d\)-invariant, since the above sections are permuted by \(S_d\).
	Taking \(S_d\)-quotients, this descends to a closed immersion \(Z \hookrightarrow X \times_k S\).
	The morphism \(Z \to S\) base changes to \(\coprod_{i=1}^d S' \to S'\), hence is finite étale of degree \(d\).

	Finally, one checks that the two constructions are mutually inverse: both are compatible with base change in \(S\), and become mutually inverse after the faithfully flat base change \(S' \to S\).

	\eqref{prop:conf_rep:smooth}: If \(X\) is smooth over \(k\), then so is the open subscheme \(\Uc_X\) of \(X^d\).
	Since \(\Uc_X \to \Uc_X/S_d\) is fppf, this implies by \stacks{05B5} that \(\Uc_X/S_d\) is smooth over \(k\).

	\eqref{prop:conf_rep:dim}: If \(X\) has pure dimension \(n\), then \(X^d\) has pure dimension \(dn\), hence so has \(\Uc_X\).
	Since \(\Uc_X \to \Uc_X/S_d\) is finite surjective, the scheme \(\Uc_X/S_d\) has pure dimension \(dn\).

	\eqref{prop:conf_rep:dense}: If \(U\) is a dense open subscheme of \(X\), then \(U^d\) is a dense open subscheme of \(X^d\).
	Indeed since each projection \(X^d \to X\) is flat, the generic points of \(X^d\) lie over the generic points of \(X\), hence in \(U^d\).
	Since \(\Delta_U=\Delta_X \cap U^d\), it follows that \(\Uc_U\) is a dense open subscheme of \(\Uc_X\).
	Faithful flatness of the quotient morphism \(\Uc_X \to \Uc_X/S_d\) implies that \(\Uc_U/S_d\) is a dense open subscheme of \(\Uc_X/S_d\), i.e.\ that \(\Conf{d}{U}\) is a dense open subscheme of \(\Conf{d}{X}\).
\end{proof}

\begin{para}
	\label{p:conf_bc}
	In the situation of (\ref{def:conf}), if \(K/k\) is a field extension, there is a canonical isomorphism of \(K\)-varieties \(\Conf{d}{X_K} \simeq \Conf{d}{X}_K\).
\end{para}

\begin{para}
	If the \(k\)-group \(G\) acts on the \(k\)-variety \(X\), then \(G\) naturally acts on \(\Conf{d}{X}\).
	Namely, for a \(k\)-scheme \(S\), the group \(G(S)\) acts on the closed subschemes of \(X_S\), preserving the property of being finite étale of degree \(d\) over \(S\).
\end{para}

\begin{para}
	\label{p:conf_equiv}
	If \(U\) is a \(G\)-invariant dense open subscheme of \(X\), then the open subscheme \(\Conf{d}{U}\) of \(\Conf{d}{X}\) (see (\ref{prop:conf_rep}.\ref{prop:conf_rep:dense})) is \(G\)-invariant.
\end{para}

\begin{lemma}
	\label{lemm:Conf_tw}
	Let \(X\) be a \(k\)-variety with a \(G\)-action, and \(T\) a \(G\)-torsor over a field \(K\supset k\).
	Then there is an isomorphism of \(K\)-varieties \(\tw{(\Conf{d}{X})}{T} \simeq \Conf{d}{\tw{X}{T}}\).
\end{lemma}
\begin{proof}
	For a \(K\)-scheme \(S\), an \(S\)-point of \(\tw{\Conf{d}{X}}{T}\) is a \(G\)-equivariant morphism \(T \times_K S \to \Conf{d}{X}\), that is a \(G\)-invariant closed subscheme of \(X \times_k T \times_K S\) which is finite étale of degree \(d\) over \(T \times_K S\).
	Since by (\ref{p:torsor_pb}) the \(G\)-torsor \(X \times_k T \times_K S \to (\tw{X}{T}) \times_K S\) is the base change of the \(G\)-torsor \(T\times_K S \to S\), such a closed subscheme corresponds by descent to a closed subscheme of \((\tw{X}{T}) \times_K S\) which is finite étale of degree \(d\) over \(S\), i.e.\ an \(S\)-point of \(\Conf{d}{\tw{X}{T}}\).
\end{proof}

\begin{lemma}
	\label{lemm:conf_fixed}
	Let \(H\) be a finite étale \(k\)-group of degree \(r \in \Nn\), and \(X \to Y\) an \(H\)-torsor, where \(Y\) is a \(k\)-variety.
	Let \(m \in \Nn\smallsetminus \{0\}\).
	Then there exists an isomorphism of \(k\)-varieties \(\Conf{m}{Y} \simeq \Conf{mr}{X}^H\).
\end{lemma}
\begin{proof}
	Let \(S\) be a \(k\)-scheme.
	An \(S\)-point of \(\Conf{mr}{X}\) belongs to \(\Conf{mr}{X}^H\) if and only if the corresponding closed subscheme of \(X_S\) is \(H\)-invariant (for the trivial \(H\)-action on \(S\)).

	Given a closed subscheme \(W\) of \(Y_S\) which is finite étale of degree \(m\) over \(S\), the pullback \(W \times_Y X\) is an \(H\)-invariant closed subscheme of \(X_S\) which is finite étale of degree \(mr\) over \(S\).
	Conversely, assume that \(Z\) is an \(H\)-invariant closed subscheme of \(X_S\) which is finite étale of degree \(mr\) over \(S\).
	Then \(Z\) descends uniquely along the \(H\)-torsor \(X_S \to Y_S\) to a closed subscheme \(W\) of \(Y_S\).
	Since the composite morphism \(Z \to W \to S\) is finite étale of degree \(mr\), and \(Z \to W\) is finite étale of degree \(r\) (being a base change of the \(H\)-torsor \(X \to Y\)), it follows that \(W \to S\) is finite étale of degree \(m\).
	The two constructions are functorial, and mutually inverse.
\end{proof}

\begin{lemma}
	\label{lemm:bounded_stab}
	Let \(X\) be a \(k\)-variety with a free \(G\)-action.
	Let \(F/k\) be a field extension.
	Then the stabilizer group scheme of any \(F\)-point of \(\Conf{d}{X}\) is a finite étale \(F\)-group, of degree dividing \(d\).
\end{lemma}
\begin{proof}
	Let \(Z \in \Conf{d}{X}(F)\).
	Let \(H=\Stab_G(Z)\), an \(F\)-group scheme.
	Then we have a morphism of \(F\)-groups \(H \to \Aut_F(Z)\), whose kernel is an \(F\)-group that we denote by \(A\).
	Pulling back along \(Z \to \Spec F\), the \(Z\)-group \(A_Z\) is a subgroup of \(G_Z\).
	Since \(G\) acts freely on \(X\), we have a cartesian square
	\[
		\begin{tikzcd}[column sep=7em]
			Z \ar[equal]{r}\ar[d, "{z \mapsto (1,z)}"']& Z \ar[d, "{z \mapsto (z,z)}"]\\
			G \times_k Z \ar[r, "{(g,z) \mapsto (gz,z)}"'] &X \times_k Z
		\end{tikzcd}
	\]
	The composite \(A_Z \to G_Z=G \times_k Z \to X \times_k Z\) factors through \(Z \subset X \times_k Z\), and the above cartesian diagram implies that \(A_Z=1\) as \(Z\)-group.
	Since \(Z \to \Spec F\) is faithfully flat, it follows that \(A=1\) as \(F\)-group.

	Thus \(H\to\Aut_F(Z)\) is a monomorphism.
	Let \(\overline{F}\) be an algebraic closure of \(F\).
	Then \(H_{\overline{F}} \to \Aut_{\overline{F}}(Z_{\overline{F}})\simeq S_d\) is a monomorphism to a finite constant \(\overline{F}\)-group.
	Since the source of a monomorphism into \(\Spec \overline{F}\) is either empty or isomorphic to \(\Spec \overline{F}\), we deduce that \(H_{\overline{F}}\) is a disjoint union of copies of \(\Spec \overline{F}\).
	In other words the \(F\)-group \(H\) is finite étale.

	Since \(G_F\) acts freely on \(X_F\), its subgroup \(H\) acts freely on \(Z \subset X_F\), hence the group \(H(\overline{F})\) acts freely on the set \(Z(\overline{F})\), whose cardinality is \(d\).
	It follows that the cardinality of \(H(\overline{F})\) divides \(d\).
	This cardinality equals the degree of the finite étale \(F\)-group \(H\).
\end{proof}

\begin{para}
	For a \(k\)-group \(G\) and an integer \(m \in \Nn\), we denote by \(G[m]\) the closed subscheme of \(G\) such that, for \(S\) a \(k\)-scheme,
	\[
		G[m](S)=\{g \in G(S) \text{ such that } g^m=1\}.
	\]
\end{para}

\begin{para}
	\label{p:G[m]_finite_mult}
	If \(G\) is either finite or of multiplicative type, then the \(k\)-scheme \(G[m]\) is finite, for every \(m \in \Nn \smallsetminus \{0\}\).
\end{para}

\begin{proposition}
	\label{prop:conf_gen_free}
	Let \(d \in \Nn \smallsetminus \{0\}\).
	Let \(X\) be a \(k\)-variety of pure dimension \(n>0\), with a generically free \(G\)-action.
	If \(d = 2\), assume that \(\dim G[2] < n\).
	Then the \(G\)-action on \(\Conf{d}{X}\) is generically free.
\end{proposition}
\begin{proof}
	Replacing \(X\) with a \(G\)-invariant dense open subscheme, by (\ref{prop:conf_rep}.\ref{prop:conf_rep:dense}) and (\ref{p:conf_equiv}) we may assume that the \(G\)-action on \(X\) is free.
	We may also assume that \(X\) is nonempty.
	Let \(C=\Conf{d}{X}\).

	Let us first note that, for any prime \(\ell\) dividing \(d\), we have
	\begin{equation}
		\label{eq:conf_gen_free}
		\dim G[\ell] < (d-d/\ell)n.
	\end{equation}
	This follows from the assumption on \(G[2]\) when \(d=2\).
	As \(G \times_k X \to X \times_k X\) is a monomorphism, we have \(\dim G +n \le 2n\), i.e.\ \(\dim G \le n\).
	Since \(\dim G[\ell] \le \dim G\), the formula \eqref{eq:conf_gen_free} follows when \(d \ne 2\) (because then \(d-d/\ell>1\)).

	Let \(\ell\) be a prime dividing \(d\).
	Consider the closed subscheme \(P[\ell] \subset (G[\ell] \smallsetminus \{1\}) \times_k C\) of pairs \((g,x)\) with \(gx=x\).
	It carries a \(G\)-action, given by \(h(g,x) = (hgh^{-1},hx)\), making the morphism \(P[\ell] \to C\) a \(G\)-equivariant morphism.
	Let \(B[\ell]\) be the scheme-theoretic image of this morphism.
	This is a \(G\)-invariant closed subscheme of \(C\):
	indeed, since \(G\to \Spec k\) is flat, formation of scheme-theoretic image commutes with product with \(G\) by \stacks{081I}, which formally implies that the action morphism \(G \times_k C \to C\) restricts to a morphism \(G \times_k B[\ell] \to B[\ell]\).

	Let \(g \in G[\ell] \smallsetminus \{1\}\) and \(K=\kappa(g)\) its residue field.
	Denote by \(P[\ell]_g\) the fiber of \(P[\ell]\) over \(g\).
	Then the action of the finite constant \(K\)-group \(J=\Zz/\ell\) via \(g\) on \(X_K\) is free, and by (\ref{p:quotient_finite}) there exists a \(K\)-variety \(Y\) and a \(J\)-torsor \(X_K \to Y\).
	The \(K\)-variety \(Y\) has pure dimension \(n\), since the same is true for \(X_K\).
	By (\ref{lemm:conf_fixed}) and (\ref{p:conf_bc}), we have \((C_K)^J \simeq \Conf{d/\ell}{Y}\), and so \(\dim (C_K)^J \le dn/\ell\).
	Since \(P[\ell]_g \to (C_K)^J\) is an isomorphism, we have \(\dim P[\ell]_g \le dn/\ell\).
	Therefore, by \eqref{eq:conf_gen_free} and the fiber dimension theorem \stacks{0B2L}, we have
	\[
		\dim B[\ell] \le \dim P[\ell] \le \dim G[\ell] +dn/\ell <dn.
	\]
	Since \(C\) has pure dimension \(dn\) by (\ref{prop:conf_rep}.\ref{prop:conf_rep:dim}), it follows that \(B[\ell]\) is nowhere dense in \(C\).

	Therefore the open subscheme of \(C\) defined as
	\[
		C^\circ = C \smallsetminus \bigcup_\ell B[\ell],
	\]
	where \(\ell\) runs over the prime numbers \(\ell\) dividing \(d\), is dense in \(C\).
	Since each \(B[\ell]\) is \(G\)-invariant, so is \(C^{\circ}\).

	For any \(x \in C\) with residue field \(F=\kappa(x)\), the group \(H=\Stab_G(x)\) is a finite étale \(F\)-group of degree dividing \(d\), by (\ref{lemm:bounded_stab}).
	Thus if \(H \ne 1\), letting \(\overline{F}\) be an algebraic closure of \(F\), the group \(H(\overline{F})\) contains an element \(g\) of order \(\ell\), for some prime \(\ell\) dividing \(d\).
	Then \((g,x) \in P[\ell](\overline{F})\), showing that its image \(x\) lies in \(B[\ell]\), and thus not in \(C^{\circ}\).
	We have proved that the stabilizer of every point of \(C^\circ\) is trivial.

	Let \(S\) be a \(k\)-scheme, and let \(z \in C^{\circ}(S)\) correspond to a closed subscheme \(Z \subset X_S\) finite étale of degree \(d\) over \(S\).
	Let \(g\in G(S)\) be such that \(gz=z\).
	Then \(g\) induces an automorphism \(\gamma\) of \(Z\) over \(S\).
	For every point \(s \in S\), we have \(g_s =1\) by the preceding argument (applied to the image of \(s\) in \(C^{\circ}\)), and so \(\gamma_s=\id\).
	Since \(Z\) is finite étale over \(S\), the equalizer of the \(S\)-morphisms \(\gamma\) and \(\id_Z\) is open and closed in \(Z\).
	It contains every point of \(Z\), hence equals \(Z\).
	We have shown that \(g\) acts trivially on \(Z\).

	After pulling back along the faithfully flat morphism \(Z\to S\), the
	element \(g\) thus fixes the tautological point of \(X_Z\).
	Since the \(G\)-action on \(X\) is free, we obtain \(g_Z=1\).
	Faithfully flat descent along \(Z\to S\) then gives \(g=1\).
	Thus the \(G\)-action on \(C^\circ\) is free, so that the \(G\)-action on \(C\) is generically free.
\end{proof}

\begin{lemma}
	\label{lemm:conf_index}
	Let \(X\) be a \(k\)-variety having a closed point of degree \(m\), whose residue field is separable over \(k\).
	Set \(d = p^{v_p(m)}\).
	Then the index of the \(k\)-variety \(\Conf{d}{X}\) is prime to \(p\).
\end{lemma}
\begin{proof}
	Let \(x\) be such a closed point, and \(E = \Spec \kappa(x)\).
	The closed immersion \(E \to X\) induces a morphism \(\Conf{d}{E} \to \Conf{d}{X}\).
	By (\ref{prop:conf_rep}) the \(k\)-scheme \(\Conf{d}{E}\) is finite étale.
	Its degree is \(\binom{m}{d}\), since \(\Conf{d}{E}(\overline{k})\) is the set of subsets of cardinality \(d\) in \(E(\overline{k})\).
	This integer is prime to \(p\) by Lucas's theorem.
	Therefore the index of the \(k\)-variety \(\Conf{d}{E}\), and a fortiori of \(\Conf{d}{X}\), is prime to \(p\).
\end{proof}

\begin{proposition}
	\label{prop:conf_versal}
	Let \(X\) be a nonempty smooth \(k\)-variety with a \(G\)-action.
	Let \(c= v_p(\ind_G(X))\).
	Then the \(G\)-action on the \(k\)-variety \(\Conf{p^c}{X}\) is \(p\)-versal.
\end{proposition}
\begin{proof}
	Let \((T_0,F_0)\) be as in (\ref{p:generic_exists}).
	The \(F_0\)-variety \(\tw{X}{T_0}\) is smooth and \(v_p(\ind(\tw{X}{T_0}))=c\).
	It is proved in \cite[Theorem~9.2]{Index} that the index of a smooth variety is the gcd of the degrees of its closed points having a residue field separable over the base field.
	Thus the \(F_0\)-variety \(\tw{X}{T_0}\) has a closed point of degree \(m\), with \(c=v_p(m)\), whose residue field is separable over \(F_0\).
	Therefore, by (\ref{lemm:Conf_tw}) and (\ref{lemm:conf_index})
	\[
		\ind(\Conf{p^c}{{\tw{X}{T_0}}}) = \ind(\tw{(\Conf{p^c}{X})}{T_0}) \not \in p \Zz,
	\]
	which means that the \(G\)-action on \(\Conf{p^c}{X}\) is weakly \(p\)-versal (see (\ref{p:trans_versal})).
	Since \(\Conf{p^c}{X}\) is smooth by (\ref{prop:conf_rep}), the \(G\)-action is \(p\)-versal by (\ref{prop:versal_weak}).
\end{proof}

\begin{theorem}
	\label{th:conf_main}
	Let \(X\) be a smooth \(k\)-variety with a generically free \(G\)-action.
	Assume that \(X\) is nonempty, and has no zero-dimensional connected component.
	Let \(d=p^c\), where \(c=v_p(\ind_G(X))\).
	If \(d=2\), assume that every connected component of \(X\) has dimension \(>\dim G[2]\) (by (\ref{p:G[m]_finite_mult}) this is so when \(G\) is either finite or of multiplicative type).
	Then
	\[
		\ed_p(G) + \dim G \le d\dim X.
	\]
\end{theorem}
\begin{proof}
	Write \(X=\coprod_j X_j\), where each \(X_j\) is \(G\)-connected (see (\ref{p:G_connected_decomp})).
	Then for any twisting pair \((T,K)\), we have \(\tw{X}{T}=\coprod_j \tw{X_j}{T}\).
	Thus \(\ind_G(X)\) is the gcd of the integers \(\ind_G(X_j)\), hence there exists \(i\) such that \(v_p(\ind_G(X_i))=c\).
	The \(k\)-variety \(X_i\) is \(G\)-connected, and thus has pure dimension \(n>0\) by (\ref{p:G_connected_equidim});
	and we have \(n \le \dim X\).
	By (\ref{prop:conf_rep}), the \(k\)-variety \(\Conf{d}{X_i}\) is smooth, and has pure dimension \(dn\).
	The \(G\)-action on \(\Conf{d}{X_i}\) is \(p\)-versal by (\ref{prop:conf_versal}), and generically free by (\ref{prop:conf_gen_free}).
	Hence (\ref{prop:ed_versal}) gives
	\[
		\ed_p(G) + \dim G \le \dim \Conf{d}{X_i} = dn \le d \dim X.\qedhere
	\]
\end{proof}

\begin{remark}
	The assumption on \(G[2]\) is necessary in (\ref{th:conf_main}).

	Indeed, assume that \(k\) has characteristic different from \(2\).
	Let \(G=\Oo(h)=\Gm \rtimes (\Zz/2)\), as in (\ref{rem:O_2_GF}).
	Then as a scheme \(G=\Gm \sqcup \sigma \Gm\), and \(\sigma \Gm \subset G[2]\), since for \(t\in \Gm(R)\), with \(R\) a \(k\)-algebra, we have \((\sigma t)^2=tt^{-1}=1\).
	Thus \(\dim G[2] = \dim G=1\).

	Let \(X=G\), with \(G\) acting by left translation.
	Let \((T_0,F_0)\) be as in (\ref{p:generic_exists}).
	Then we have an isomorphism of \(F_0\)-varieties \(\tw{X}{T_0} \xrightarrow{\sim} T_0\) given by \((x,t) \mapsto x^{-1}t\), and so \(\ind_G(X)=\ind(T_0)\).
	Let \(K\) be a quadratic étale \(F_0\)-algebra splitting the image of the class of \(T_0\) under the morphism \(H^1(F_0,G) \to H^1(F_0,\Zz/2)\).
	Then the class of \((T_0)_K\) lies in the image of \(H^1(K, \Gm) \to H^1(K, G)\), which vanishes by Hilbert's Theorem 90.
	Therefore \(T_0(K)\ne \varnothing\), and thus \(\ind(T_0) \mid 2\).
	Setting \(c=v_2(\ind_G(X))\), and \(d=2^c\), we thus have \(d \le 2\).

	On the other hand, we have \(\ed_2(G) = \ed_2(\Oo(h)) = 2\) by \cite[Example~2.5]{Merkurjev-Bourbaki}.
	Thus
	\[
		\ed_2(G) + \dim G= 3 > 2 \ge d \dim X.
	\]
	Note that, in fact \(d= 2\), since otherwise \(d=1\) so that \(X\) would be \(2\)-versal by (\ref{p:trans_versal}) and (\ref{prop:versal_weak}), contradicting (\ref{prop:ed_versal}) (which has no assumption on \(G[2]\)).
\end{remark}

\section{\texorpdfstring{\(p\)}{p}-adic valuation of Chern numbers}
\label{sect:p-adic}
In this section we illustrate how the method of \S\ref{sect:conf} can be used to strengthen the results of \S\S\ref{sect:noncommutative}--\ref{sect:derived}.

\begin{theorem}
	\label{th:GL_p-adic}
	Let \(G\) be a finite constant \(p\)-group, viewed as a \(k\)-group, and assume that the characteristic of \(k\) differs from \(p\).
	Let \(X\) be a smooth projective \(k\)-variety with a faithful \(G\)-action.
	Set \(c=v_p(\ind_G(X))\) and \(n=\dim X\), and assume that \(n>0\).
	Suppose that one of the following conditions holds:
	\begin{enumerate}[label=(\roman*), ref=\roman*]
		\item
			\label{cond:geom_conn}
			the \(k\)-variety \(X\) is geometrically connected,

		\item
			\label{cond:G_conn}
			the \(k\)-group \(G\) is abelian, and \(X\) is \(G\)-connected.
	\end{enumerate}
	Then \(G(k)\) is a subgroup of \(\GL_{p^cn}(k(\zeta_p))\).
\end{theorem}
\begin{proof}
	The \(G\)-action on \(X\) is faithful, hence generically free, by (\ref{p:finite_genfree}) in case \eqref{cond:geom_conn} (as \(X\) is then geometrically integral), resp.\ (\ref{lemm:GF_abelian}) in case \eqref{cond:G_conn} (in which case \(X\) has pure dimension \(n>0\) by (\ref{p:G_connected_equidim})).
	Since \(\dim G[2]=0 < n\) by (\ref{p:G[m]_finite_mult}), it follows from (\ref{th:conf_main}) that \(\ed_p(G) \le p^cn\).
	As \(\ed_p(G)\ge \ed_p(G_{k(\zeta_p)})\) by \cite[Proposition~1.5]{Merkurjev-essential_contemp}, we have \(\ed_p(G_{k(\zeta_p)}) \le p^cn\).
	By \cite[Theorem~4.1]{KM-Ess}, the group \(G(k)\) has a faithful representation over \(k(\zeta_p)\) of dimension equal to \(\ed_p(G_{k(\zeta_p)}) \le p^cn\), hence it is a subgroup of \(\GL_{p^cn}(k(\zeta_p))\).
\end{proof}

As in (\ref{prop:Minkowski}), we deduce:
\begin{corollary}
	\label{prop:Minkowski_p-adic}
	In the situation of (\ref{th:GL_p-adic}), we have \(v_p(|G|) \le M_p(p^cn,k(\zeta_p))\).
	Thus, if \(p\) is odd, or \(p=2\) with \(\zeta_4 \in k\) (see (\ref{p:m_F})),
	\[
		v_p(|G|) \le m_{k(\zeta_p)} \cdot p^c \cdot n + v_p((p^cn)!).
	\]
\end{corollary}

\begin{proposition}
	\label{th:P_p-adic}
	Let \(\Pc\) be a class of finite groups (see \S\ref{sect:obstruction}).
	In the situation of (\ref{th:GL_p-adic}), if \(p^cn <d_p^{\Pc}(G)\), then \(G(k) \in \Pc\).
\end{proposition}
\begin{proof}
	As in the proof of (\ref{th:GL_p-adic}), we have \(\ed_p(G) \le p^cn\), hence \(\ed_p(G) < d_p^{\Pc}(G)\), so that \(G(k)=G(\overline{k}) \in \Pc\).
\end{proof}

\begin{corollary}
	\label{cor:derived_p-adic}
	In the situation of (\ref{th:GL_p-adic}), the group \(G(k)\) has derived length \(\le c+1+\lfloor \log_pn \rfloor\).
\end{corollary}
\begin{proof}
	This follows from (\ref{th:P_p-adic}) and (\ref{rem:d_p_der_power}).
\end{proof}

\begin{proposition}
	\label{prop:ab_p-adic}
	In the situation of (\ref{th:GL_p-adic}), assume that \(G=(\Zz/p^{n_1}) \times\dotsb \times (\Zz/p^{n_r})\) for some \(n_1,\dotsc,n_r \in \Nn \smallsetminus \{0\}\).
	Then
	\[
		[k(\zeta_{p^{n_1}}):k(\zeta_p)] + \dotsb + [k(\zeta_{p^{n_r}}):k(\zeta_p)] \le p^cn.
	\]
\end{proposition}
\begin{proof}
	This follows from (\ref{th:GL_p-adic}) and (\ref{lemm:faithful_abelian})  applied over \(k(\zeta_p)\).
\end{proof}

\begin{corollary}
	\label{cor:exponent-p-adic}
	Assume that the characteristic of \(k\) differs from \(p\).
	Let \(X\) be a smooth projective \(k\)-variety, and \(g\) an automorphism of \(X\) whose order is \(p^s\), for some \(s\in \Nn\smallsetminus\{0\}\).
	Assume that \(\dim X>0\), that \(g\) permutes transitively the connected components of \(X\), and that \(X\) has a Chern number not divisible by \(p^{c+1}\), where \(c\in \Nn\).
	Then \([k(\zeta_{p^s}):k(\zeta_p)] \le p^c\dim X\).
\end{corollary}
\begin{proof}
	Let \(G\) be the finite constant \(k\)-group \(\Zz/p^s\), acting on \(X\) via \(g\).
	The \(k\)-variety \(X\) is \(G\)-connected by (\ref{p:G_connected_transitive}), and \(v_p(\ind_G(X)) \le c\) by (\ref{prop:Chern_versal}.\ref{prop:Chern_versal:1}).
	Thus the statement follows from (\ref{prop:ab_p-adic}).
\end{proof}

\begin{remark}
	Since \([\Qq(\zeta_{p^s}):\Qq(\zeta_p)]=p^{s-1}\) for \(s \ge 1\), Corollary~(\ref{cor:exponent-p-adic}) for \(k=\Qq\) reads
	\[
		s \le c+1 + \lfloor \log_p(\dim X) \rfloor.
	\]
\end{remark}

The next remarks illustrate the sharpness of (\ref{cor:exponent-p-adic}).
\begin{remark}
	\label{rem:sharp_p-adic_odd}
	Assume that \(p\) is odd, and let \(c\in \Nn\smallsetminus \{0\}\).
	Set \(k=\Qq(\zeta_p)\), and \(X=\Pp^1_{k(\zeta_{p^{c+1}})}\) viewed as a \(k\)-variety.
	Then \(X\) is connected, with \([k(\zeta_{p^{c+1}}):k]=p^c\) geometric connected components.
	This variety has the Chern number \(\deg c_1(\Tan_X)=2p^c\), which is not divisible by \(p^{c+1}\).

	The automorphism \(g\) of the \(k\)-variety \(X\) given by \([x:y] \mapsto [\zeta_{p^{c+1}} \cdot x:y]\) has order \(p^{c+1}\).
	We have
	\[
		[k(\zeta_{p^{c+1}}):k(\zeta_p)] = p^c= p^c\dim X.
	\]
\end{remark}

\begin{remark}
	\label{rem:sharp_p-adic_2}
	Assume now that \(p=2\), and let \(c\in \Nn\).
	Set \(k=\Qq(i)=\Qq(\zeta_4)\), and \(K=\Qq(\zeta_{2^{c+2}})\).
	Let \(X=\Pp^2_K\) viewed as a \(k\)-variety.
	Then \(X\) is connected, with \([K:k]=2^c\) geometric connected components.
	This variety has the Chern number \(\deg c_2(\Tan_X)=3\cdot 2^c\), which is not divisible by \(2^{c+1}\).

	Let \(V=K(\zeta_{2^{c+3}}) \times K\), a \(K\)-vector space of dimension \(3\).
	The automorphism of \(V\) given by \((x,y) \mapsto (\zeta_{2^{c+3}} \cdot x,y)\), for \(x\in K(\zeta_{2^{c+3}}), y \in K\), induces an automorphism \(g\) of the \(K\)-variety (and thus also of the \(k\)-variety) \(\Pp_K(V)\simeq X\).
	Then \(g\) has order \(2^{c+3}\), and
	\[
		[k(\zeta_{2^{c+3}}):k(\zeta_2)] = [\Qq(\zeta_{2^{c+3}}):\Qq(i)]= 2^{c+1}=p^c \dim X.
	\]
\end{remark}

\begin{remark}
	\label{rem:sharp_p-adic_2:geom_connected}
	Here is an example with \(X\) geometrically connected and \(c=1\).
	Assume that \(p=2\).
	Let \(k=\Qq\) and \(X=\Pp^1\).
	Then \(X\) has a Chern number equal to \(2\), which is not divisible by \(p^{c+1}\).
	Let \(g\) be the automorphism of \(X\) given by \([x:y] \mapsto [x-y:x+y]\).
	Then \(g\) has order \(4\), and
	\[
		[k(\zeta_4):k(\zeta_2)]=[\Qq(i):\Qq]=2 = p^c \dim X.
	\]
\end{remark}

\begin{para}
	\label{p:Todd}
	Let \(X\) be a smooth projective \(k\)-variety of dimension \(n\).
	Then \(X\) admits a Chern number \(e\) such that
	\[
		v_p(e) \le v_p(\chi(X,\Oc_X)) + \Big\lfloor \frac{n}{p-1} \Big\rfloor.
	\]
	This follows from the Hirzebruch--Riemann--Roch formula, and the computation of the \(p\)-adic valuation of the denominator in the Todd class (see e.g.\ \cite[Theorem~5.1]{invariants}).
\end{para}

As in \S\ref{sect:birat}, one can combine (\ref{p:Todd}) and (\ref{p:regularization}) with (\ref{th:GL_p-adic}) and its corollaries to obtain restrictions on \(p\)-groups of birational automorphisms.
Let us spell out the consequences for the Cremona group:

\begin{proposition}
	\label{cor:Cremona_p-adic}
	Let \(n\in \Nn\smallsetminus \{0\}\), and set \(u=n \cdot p^{\lfloor n/(p-1) \rfloor}\).
	Assume that \(k\) has characteristic zero.
	Then
	\begin{enumerate}[label=(\roman*), ref=\roman*]
		\item 
			\label{cor:Cremona_p-adic:exponent}
			the group \(\Cr_n(k)\) contains no element of order \(p^s\), when \(s\ge 1\) is such that \([k(\zeta_{p^s}):k(\zeta_p)] > u\).
	\end{enumerate}
	In addition, if \(G\) is a finite \(p\)-subgroup of \(\Cr_n(k)\), then
	\begin{enumerate}[label=(\roman*), ref=\roman*, resume]
		\item 
			\label{cor:Cremona_p-adic:GL}
			\(G \subset \GL_u(k(\zeta_p))\),
		\item
			\label{cor:Cremona_p-adic:MS}
			\(v_p(|G|) \le M_p(u,k(\zeta_p))\),
		\item 
			\(G\) has derived length \(\le 1+\lfloor \log_pn \rfloor +\lfloor n/(p-1) \rfloor\).
	\end{enumerate}
\end{proposition}
\begin{proof}
	In case \eqref{cor:Cremona_p-adic:exponent}, let \(g\) be an element of order \(p^s\), and \(G\simeq \Zz/p^s\) the subgroup of \(\Cr_n(k)\) generated by \(g\).
	In any case, we view \(G\) as a finite constant \(k\)-group.
	Let \(X=\Pp^n\), and let \(Y\) be as in (\ref{p:regularization}), so that \(G\) acts faithfully on \(Y\), and \(\chi(Y,\Oc_Y)=\chi(X,\Oc_X)=1\).
	By (\ref{p:Todd}), the \(k\)-variety \(Y\) admits a Chern number not divisible by \(p^{1+\lfloor n/(p-1) \rfloor}\), and thus \(v_p(\ind_G(Y)) \le \lfloor n/(p-1) \rfloor\) by (\ref{prop:Chern_versal}.\ref{prop:Chern_versal:1}).

	The statements then follow respectively from (\ref{cor:exponent-p-adic}), (\ref{th:GL_p-adic}), (\ref{prop:Minkowski_p-adic}) and (\ref{cor:derived_p-adic}).
\end{proof}

\begin{remark}
	\label{rem:cr_p-adic}
	When \(k=\Qq\) the condition of (\ref{cor:Cremona_p-adic}.\ref{cor:Cremona_p-adic:exponent}) reads
	\[
		s>1 + \Big\lfloor \frac{n}{p-1} \Big\rfloor + \lfloor \log_pn \rfloor.
	\]
	For instance, when \(n=1\) and \(p=2\), we obtain that \(\Cr_1(\Qq)\) has no element of order \(8\).
	Recall that \(\PGL_2(\Qq) \subset \Cr_1(\Qq)\) has an element of order \(4\) by (\ref{rem:sharp_p-adic_2:geom_connected}).
\end{remark}

\begin{remark}
	Assume that \(k\) is finitely generated over \(\Qq\).
	The orders of finite subgroups of \(\Cr_n(k)\) are then known to be bounded (\cite[Theorem~1.4]{Prokhorov-Shramov}, combined with \cite{BAB}), but no explicit bound is known in general (see \cite{Serre-Minkowski} for \(n=2\) and \cite[Theorem~1.4]{Zaitsev} for \(n=3\)).
	To the best of our knowledge, the statements (\ref{cor:Cremona_p-adic}.\ref{cor:Cremona_p-adic:MS}) and (\ref{cor:Cremona_p-adic}.\ref{cor:Cremona_p-adic:exponent}) provide the first explicit bounds valid for all \(n\) on the \(p\)-part of orders and exponents of such subgroups.
	We do not expect them to be sharp, in general.

	Similarly, every finite subgroup of \(\Cr_n(k)\) is a subgroup of \(\GL_N(k)\), for some integer \(N\) which depends only on \(n\) and \(k\) (see \cite[Theorem~1.4]{DHU}), but explicit values of \(N\) for all \(n\) appear to be out of reach \cite[\S1]{DHU}.
	The statement (\ref{cor:Cremona_p-adic}.\ref{cor:Cremona_p-adic:GL}) provides the value \(N=[k(\zeta_p):k]\cdot u\) for finite \(p\)-subgroups (by restriction of scalars).
	By \cite[Theorem~1.5]{DHU}, the minimal such \(N\) grows at least exponentially with \(n\);
	the upper bound that we obtain in (\ref{cor:Cremona_p-adic}), for \(p\)-subgroups as opposed to arbitrary finite subgroups, also carries the exponential factor \(p^{\lfloor n/(p-1) \rfloor}\).
\end{remark}

To conclude, let us mention that the fixed-point theorem (\ref{th:fpt}) does not immediately generalize in the same way, because its proof uses  \cite[(4.4)]{fpt}, which separately requires a Chern number prime to \(p\).

\appendix
\section{Representations of finite \texorpdfstring{\(p\)}{p}-groups}
\label{sect:rep_p}
In this appendix, we collect a few lemmas that are needed in the paper.
We let \(G\) be a finite group such that \(|G|\) is invertible in the field \(k\), viewed as a finite constant \(k\)-group.

\begin{lemma}
	\label{lemm:Clifford}
	Assume that \(k\) is algebraically closed, and that \(G\) is a finite \(p\)-group.
	Let \(V\) be an irreducible \(k[G]\)-module such that \(\dim_k V >1\).
	Then there exists a normal subgroup \(H\) of index \(p\) in \(G\), and an irreducible \(k[H]\)-module \(W\) such that \(V\simeq \Ind^G_HW\).
\end{lemma}
\begin{proof}
	Since representations of \(p\)-groups over an algebraically closed field of characteristic \(\ne p\) are monomial \cite[(11.2)]{Curtis-Reiner-Methods-I}, there exist a subgroup \(C\subset G\) and a \(k[C]\)-module \(N\) such that \(V \simeq \Ind^G_CN\) and \(\dim_kN=1\).
	Since \(\dim_k V >1\) we have \(C \ne G\);
	let \(H\) be a maximal subgroup of \(G\) containing \(C\), and set \(W =\Ind^H_C N\).
	Then \(V \simeq \Ind^G_H W\) by transitivity of the induction, and the irreducibility of the \(k[G]\)-module \(V\) implies that of the \(k[H]\)-module \(W\).
	The subgroup \(H\) is normal of index \(p\) in \(G\), being a maximal subgroup in a \(p\)-group.
\end{proof}

\begin{para}
	\label{p:ab_def}
	Let us denote by \([G,G]\) the commutator subgroup of \(G\).
	Let \(A\) be a commutative \(k\)-algebra.
	For an \(A[G]\)-module \(M\), we will denote by \(M^{[G,G]} \subset M\) the \(A[G]\)-submodule of \([G,G]\)-invariants.
	This is the largest submodule of \(M\) on which \(G\) acts through an abelian quotient.

	Consider the element (recall that the characteristic of \(k\) does not divide \(|G|\))
	\[
		\gamma = \Big|[G,G]\Big|^{-1} \Big(\sum_{g \in [G,G]} g\Big) \in A[G].
	\]
	This is a central idempotent of \(A[G]\), since \([G,G]\) is normal in \(G\).
	Then \(M^{[G,G]}=\gamma M\), and thus \(M^{[G,G]}\) is a direct summand of \(M\) as an \(A[G]\)-module.
	Therefore if \(A \to B\) is a morphism of commutative \(k\)-algebras, we have \((M\otimes_A B)^{[G,G]}=M^{[G,G]} \otimes_A B\).
\end{para}

\begin{lemma}
	\label{lemm:inv_ab}
	Let \(A\) be a commutative \(k\)-algebra, and \(M\) an \(A[G]\)-module, invertible as an \(A\)-module.
	Then \(M^{[G,G]}=M\).
\end{lemma}
\begin{proof}
	Since \(M\) is invertible we have \(\End_A(M)=A\), a commutative ring.
	The morphism \(G \to \Aut_A(M)=A^{\times}\) thus factors through the abelianization of \(G\), meaning that the group \([G,G]\) acts trivially on \(M\).
\end{proof}

\begin{para}
	\label{p:ab_dim_one}
	For a \(k[G]\)-module \(V\), we have \(V^{[G,G]}=0\) if and only if \(V_{\overline{k}}\) has no \(\overline{k}[G]\)-submodule of dimension one over \(\overline{k}\).

	Indeed one implication follows from (\ref{lemm:inv_ab}) and the base change formula of (\ref{p:ab_def}).
	On the other hand, any \(\overline{k}[G]\)-module on which \(G\) acts through its abelianization \(G/[G,G]\) splits as a direct sum of summands of dimension \(1\).
\end{para}

\begin{lemma}
	\label{lemm:fixed_proj}
	Let \(A\) be a commutative finitely generated \(k\)-algebra, and \(V\) an \(A[G]\)-module which is locally free of finite rank as an \(A\)-module.
	Then the \(A\)-module \(V^{[G,G]}\) is locally free of finite rank, and \(\Pp_A(V)^G=\Pp_A(V^{[G,G]})^G\).
\end{lemma}
\begin{proof}
	Since the \(A\)-module \(V^{[G,G]}\) is a direct summand of \(V\), it is locally free of finite rank.
	Let \(R\) be an \(A\)-algebra.
	An \(R\)-point of \(\Pp_A(V)\) is a direct summand \(M\) of the \(R\)-module \(V \otimes_A R\), which is invertible as an \(R\)-module.
	The point lies in \(\Pp_A(V)^G\) if and only if \(M\) is an \(R[G]\)-submodule of \(V \otimes_A R\).
	By (\ref{lemm:inv_ab}) we have
	\[
		M=M^{[G,G]}\subset (V \otimes_A R)^{[G,G]}=V^{[G,G]} \otimes_A R.
	\]
	Moreover, since the \(R\)-module \(M\) is a direct summand of \(V \otimes_AR\), it is a direct summand of \(V^{[G,G]} \otimes_A R\).
	This means that the \(R\)-point of \(\Pp_A(V)^G\) factors through \(\Pp_A(V^{[G,G]})\).
	This proves one inclusion, the other is clear as \(\Pp_A(V^{[G,G]})\) is a \(G\)-invariant closed subscheme of \(\Pp_A(V)\).
\end{proof}

Finally, we will need to perform a standard analysis of irreducible representations of \(p\)-groups, so let us fix the setting:
\begin{para}
	\label{p:V_irred}
	Assume that \(G\) is a \(p\)-group and that \(\zeta_p \in k\).
	Let \(V\) be an irreducible \(k[G]\)-module, and let \(D=\End_{k[G]}(V)\).
	This is a division \(k\)-algebra; its center is a field \(K\) containing \(k\).

	Let \(e=p^a\) be the exponent of \(G\), and let \(F\) be a splitting field of the polynomial \(T^e-1\in k[T]\), a separable polynomial since the characteristic of \(k\) is not \(p\).
	The field extension \(F/k\) is Galois.
	The fact that \(\zeta_p \in k\) implies that \([F:k]\) divides \(p^{\max(a-1,0)}\), hence is a power of \(p\).

	Since \(F\) contains enough roots of unity, it is a splitting field of \(G\) by \cite[(17.1)]{Curtis-Reiner-Methods-I};
	therefore the \(F[G]\)-module \(V_F\) splits as a direct sum of absolutely irreducible \(F[G]\)-modules.
	This implies that the center of \(\End_{F[G]}(V_F)\) is a split étale \(F\)-algebra.
	Since it is isomorphic to \(K\otimes_k F\), we conclude that \(K/k\) may be embedded as a subextension of \(F/k\).
	As \(F/k\) is Galois abelian of degree a power of \(p\), the same is true for \(K/k\).

	Let \(W\) be an irreducible summand of the \(F[G]\)-module \(V_F\).
	The sum of its \(\Gal(F/k)\)-conjugates in \(V_F\) is a \(\Gal(F/k)\)-invariant \(F[G]\)-submodule of \(V_F\), which is thus of the form \(U_F\) for a \(k[G]\)-submodule \(U\) of \(V\);
	the irreducibility of \(V\) forces \(U=V\), and so \(U_F=V_F\).
	This implies that all irreducible summands of the \(F[G]\)-module \(V_F\) are isomorphic to some \(\Gal(F/k)\)-conjugate of \(W\); in addition they appear with a common multiplicity \(m\), by symmetry under the \(\Gal(F/k)\)-action which permutes the isotypic components transitively.
	Thus there exist pairwise nonisomorphic irreducible \(F[G]\)-modules \(W_1,\dotsc, W_s\), conjugate to one another under the \(\Gal(F/k)\)-action, and such that
	\begin{equation}
		\label{eq:dim_V_F}
		V_F \simeq (W_1 \oplus \dotsb \oplus W_s)^{\oplus m}.
	\end{equation}
	Since \(W_i\) is an absolutely irreducible \(F[G]\)-module, we have \(\End_{F[G]}(W_i)=F\), for each \(i=1,\dotsc,s\).
	Therefore \(D \otimes_k F\simeq \End_{F[G]}(V_F) \simeq (M_m(F))^s\).
	Its center \(K\otimes_k F\) is thus isomorphic to \(F^s\), and so \(s=[K:k]\).
	Taking the dimensions over \(F\), we also obtain that \(\dim_K D=m^2\), so \(m\) is the Schur index of the central division \(K\)-algebra \(D\).
	Since the \(F[G]\)-modules \(W_i\) are absolutely irreducible, their common dimension \(w\) is a divisor of \(|G|\) (see e.g.\ \cite[VIII, \S21, n$^\circ$12, Corollaire~2, p.410]{Bou-A-8}), and so \(w\) is a power of \(p\).
	As \(V\simeq D^{\oplus r}\) as a \(D\)-module for some \(r\), we have \(\dim_k V=rsm^2\).
	But \(\dim_k V=\dim_F V_F=swm\) by \eqref{eq:dim_V_F}.
	We deduce that \(rm=w\), and so \(m\), as well as \(\dim_k V\), are powers of \(p\).
\end{para}

\def\cprime{$'$}


\begin{thebibliography}{LMMR13b}

\bibitem[Bir21]{BAB}
Caucher Birkar.
\newblock Singularities of linear systems and boundedness of {F}ano varieties.
\newblock {\em Ann. of Math. (2)}, 193(2):347--405, 2021.

\bibitem[BLR90]{Neron}
Siegfried Bosch, Werner L\"{u}tkebohmert, and Michel Raynaud.
\newblock {\em N\'{e}ron models}, volume~21 of {\em Ergebnisse der Mathematik und ihrer Grenzgebiete (3) [Results in Mathematics and Related Areas (3)]}.
\newblock Springer-Verlag, Berlin, 1990.

\bibitem[Bou12]{Bou-A-8}
N.~Bourbaki.
\newblock {\em \'{E}l\'ements de math\'ematique. {A}lg\`ebre. {C}hapitre 8. {M}odules et anneaux semi-simples}.
\newblock Springer, Berlin, 2012.
\newblock Second revised edition of the 1958 edition [MR0098114].

\bibitem[CR87]{Curtis-Reiner-Methods-II}
Charles~W. Curtis and Irving Reiner.
\newblock {\em Methods of representation theory. {V}ol. {II}}.
\newblock Pure and Applied Mathematics (New York). John Wiley \& Sons, Inc., New York, 1987.
\newblock With applications to finite groups and orders, A Wiley-Interscience Publication.

\bibitem[CR90]{Curtis-Reiner-Methods-I}
Charles~W. Curtis and Irving Reiner.
\newblock {\em Methods of representation theory. {V}ol. {I}}.
\newblock Wiley Classics Library. John Wiley \& Sons, Inc., New York, 1990.
\newblock With applications to finite groups and orders, Reprint of the 1981 original, A Wiley-Interscience Publication.

\bibitem[CR11]{Higherdirect}
Andre Chatzistamatiou and Kay R{\"u}lling.
\newblock Higher direct images of the structure sheaf in positive characteristic.
\newblock {\em Algebra Number Theory}, 5(6):693--775, 2011.

\bibitem[Deb01]{Debarre-book}
Olivier Debarre.
\newblock {\em Higher-dimensional algebraic geometry}.
\newblock Universitext. Springer-Verlag, New York, 2001.

\bibitem[DG11]{SGA3-1}
Michel Demazure and Alexandre Grothendieck.
\newblock {\em Sch\'emas en groupes ({SGA} 3, {T}ome {I}), {P}ropri\'et\'es g\'en\'erales des sch\'emas en groupes}.
\newblock Documents Math\'ematiques (Paris), 7. Soci\'et\'e Math\'ematique de France, Paris, 2011.
\newblock S{\'e}minaire de G{\'e}om{\'e}trie Alg{\'e}brique du Bois Marie 1962--64, revised and annotated edition of the 1970 French original.

\bibitem[DHU25]{DHU}
Alexander Duncan, Bailey Heath, and Christian Urech.
\newblock Representations of finite subgroups of {C}remona groups.
\newblock {\em Preprint}, 2025.
\newblock \href{https://arxiv.org/abs/2507.04474}{\texttt{arXiv:2507.04474}}.

\bibitem[DR15]{Duncan-Reichstein-versal}
Alexander Duncan and Zinovy Reichstein.
\newblock Versality of algebraic group actions and rational points on twisted varieties.
\newblock {\em J. Algebraic Geom.}, 24(3):499--530, 2015.
\newblock With an appendix containing a letter from J.-P. Serre.

\bibitem[EG98]{EG-Equ}
Dan Edidin and William Graham.
\newblock Equivariant intersection theory.
\newblock {\em Invent. Math.}, 131(3):595--634, 1998.

\bibitem[EKM08]{EKM}
Richard Elman, Nikita~A. Karpenko, and Alexander~S. Merkurjev.
\newblock {\em {The algebraic and geometric theory of quadratic forms.}}
\newblock Providence, RI: American Mathematical Society, 2008.

\bibitem[Flo08]{Forence-cyclic}
Mathieu Florence.
\newblock On the essential dimension of cyclic {$p$}-groups.
\newblock {\em Invent. Math.}, 171(1):175--189, 2008.

\bibitem[GLL13]{Index}
Ofer Gabber, Qing Liu, and Dino Lorenzini.
\newblock The index of an algebraic variety.
\newblock {\em Invent. Math.}, 192(3):567--626, 2013.

\bibitem[GMS03]{Garibaldi-Merkurjev-Serre}
Skip Garibaldi, Alexander Merkurjev, and Jean-Pierre Serre.
\newblock {\em Cohomological invariants in {G}alois cohomology}, volume~28 of {\em University Lecture Series}.
\newblock American Mathematical Society, Providence, RI, 2003.

\bibitem[Gro66]{ega-4-3}
Alexander Grothendieck.
\newblock \'{E}l\'ements de g\'eom\'etrie alg\'ebrique. {IV}. \'{E}tude locale des sch\'emas et des morphismes de sch\'emas. {III}.
\newblock {\em Inst. Hautes \'Etudes Sci. Publ. Math.}, (28):5--255, 1966.

\bibitem[Hau13]{invariants}
Olivier Haution.
\newblock Invariants of upper motives.
\newblock {\em Doc. Math.}, 18:1555--1572, 2013.

\bibitem[Hau19]{fpt}
Olivier Haution.
\newblock Fixed point theorems involving numerical invariants.
\newblock {\em Compos. Math.}, 155(2):260--288, 2019.

\bibitem[KM08]{KM-Ess}
Nikita~A. Karpenko and Alexander~S. Merkurjev.
\newblock Essential dimension of finite {$p$}-groups.
\newblock {\em Invent. Math.}, 172(3):491--508, 2008.

\bibitem[KZ26]{Kollar-Zhuang}
J\'{a}nos Koll\'{a}r and Ziquan Zhuang.
\newblock Essential dimension of isogenies.
\newblock {\em Pure Appl. Math. Q.}, 22(2):675--690, 2026.

\bibitem[LMMR13a]{LMMR-tori}
Roland L\"{o}tscher, Mark MacDonald, Aurel Meyer, and Zinovy Reichstein.
\newblock Essential dimension of algebraic tori.
\newblock {\em J. Reine Angew. Math.}, 677:1--13, 2013.

\bibitem[LMMR13b]{LMMR-ext}
Roland L\"{o}tscher, Mark MacDonald, Aurel Meyer, and Zinovy Reichstein.
\newblock Essential {$p$}-dimension of algebraic groups whose connected component is a torus.
\newblock {\em Algebra Number Theory}, 7(8):1817--1840, 2013.

\bibitem[Mer09]{Merkurjev-essential_contemp}
Alexander~S. Merkurjev.
\newblock Essential dimension.
\newblock In {\em Quadratic forms---algebra, arithmetic, and geometry}, volume 493 of {\em Contemp. Math.}, pages 299--325. Amer. Math. Soc., Providence, RI, 2009.

\bibitem[Mer16]{Merkurjev-Bourbaki}
Alexander~S. Merkurjev.
\newblock Essential dimension.
\newblock {\em Ast\'{e}risque}, (380):Exp. No. 1101, 423--448, 2016.

\bibitem[PS14]{Prokhorov-Shramov}
Yuri Prokhorov and Constantin Shramov.
\newblock Jordan property for groups of birational selfmaps.
\newblock {\em Compos. Math.}, 150(12):2054--2072, 2014.

\bibitem[RY00]{Reichstein-Youssin}
Zinovy Reichstein and Boris Youssin.
\newblock Essential dimensions of algebraic groups and a resolution theorem for {$G$}-varieties.
\newblock {\em Canad. J. Math.}, 52(5):1018--1056, 2000.
\newblock With an appendix by J{\'a}nos Koll{\'a}r and Endre Szab{\'o}.

\bibitem[Ser07]{Serre-Gk}
Jean-Pierre Serre.
\newblock Bounds for the orders of the finite subgroups of ${G}(k)$.
\newblock In Meinolf Geck, Donna Testerman, and Jacques Th\'evenaz, editors, {\em Group representation theory}, pages 405--450. EPFL Press, Lausanne, 2007.
\newblock Also available as arXiv:1011.0346.

\bibitem[Ser09]{Serre-Minkowski}
Jean-Pierre Serre.
\newblock A {M}inkowski-style bound for the orders of the finite subgroups of the {C}remona group of rank 2 over an arbitrary field.
\newblock {\em Mosc. Math. J.}, 9(1):183--198, 2009.

\bibitem[{Sta}]{stacks}
The {Stacks Project Authors}.
\newblock {\em {S}tacks {P}roject}.
\newblock \url{https://stacks.math.columbia.edu}.

\bibitem[Tho86]{Thomason-comparison}
Robert~W. Thomason.
\newblock Comparison of equivariant algebraic and topological {$K$}-theory.
\newblock {\em Duke Math. J.}, 53(3):795--825, 1986.

\bibitem[Tot99]{Totaro-CHBG}
Burt Totaro.
\newblock The {C}how ring of a classifying space.
\newblock In {\em Algebraic {$K$}-theory ({S}eattle, {WA}, 1997)}, volume~67 of {\em Proc. Sympos. Pure Math.}, pages 249--281. Amer. Math. Soc., Providence, RI, 1999.

\bibitem[Xu20]{Xu-rank_bir}
Jinsong Xu.
\newblock A remark on the rank of finite $p$-groups of birational automorphisms.
\newblock {\em Comptes Rendus. Math\'ematique}, 358(7):827--829, 2020.

\bibitem[Zai26]{Zaitsev}
Alexandr Zaitsev.
\newblock Finite subgroups of {C}remona group of rank $3$ over the field of rational numbers.
\newblock {\em Preprint}, 2026.
\newblock \href{https://arxiv.org/abs/2602.06162}{\texttt{arXiv:2602.06162}}.

\end{thebibliography}
\end{document}